\documentclass[review]{elsarticle}
\usepackage{bm,amssymb,amsmath}
\usepackage{esint}
\usepackage{marvosym}
\usepackage[all]{xy}
\usepackage[colorlinks]{hyperref}
\usepackage{multirow,subfigure}
\usepackage{graphicx,epstopdf}
\usepackage[letterpaper,margin=1in]{geometry}
\usepackage{multicol}

\newcommand{\vertiii}[1]{{\left\vert\kern-0.25ex\left\vert\kern-0.25ex\left\vert #1 
    \right\vert\kern-0.25ex\right\vert\kern-0.25ex\right\vert}}

\begin{document}

\newtheorem{remark}{Remark}
\newtheorem{definition}{Definition}
\newtheorem{corollary}{Corollary}
\newtheorem{lemma}{Lemma}
\newtheorem{theorem}{Theorem}
\newtheorem{proof}{Proof}

\numberwithin{lemma}{section}
\numberwithin{theorem}{section}
\numberwithin{remark}{section}
\numberwithin{corollary}{section}
\numberwithin{equation}{section}
\def\proof{\par\noindent{\bf Proof.\ } \ignorespaces}
\def\endproof{}
\begin{frontmatter}

\title{New stabilized $P_1\times P_0$ finite element methods for nearly inviscid and incompressible flows}

\author[mainaddress1]{Yuwen Li\corref{mycorrespondingauthor}}
\cortext[mycorrespondingauthor]{Corresponding author: yuwenli925@gmail.com (Yuwen Li)}

\author[mainaddress2]{Ludmil T Zikatanov}\ead{ludmil@psu.edu}

\address[mainaddress1,mainaddress2]{Department of Mathematics, The Pennsylvania State
University, University Park, PA 16802, USA}  

\begin{abstract}
This work proposes a new stabilized $P_1\times P_0$ finite element method for solving the incompressible Navier--Stokes equations. The  numerical scheme is based on a reduced Bernardi--Raugel element with statically condensed face bubbles and is pressure-robust in the small viscosity regime. For the Stokes problem, an error estimate  uniform with respect to the kinematic viscosity  is shown. For the Navier--Stokes equation, the nonlinear convection term is discretized using an edge-averaged finite element method. In comparison with classical schemes, the proposed method does not require tunning of parameters and is validated for competitiveness on several benchmark problems in 2 and 3 dimensional space. 
\end{abstract}


\begin{keyword}
incompressible Navier--Stokes equation\sep Bernardi--Raugel element\sep stabilized $P_1\times P_0$ element \sep edge-averaged finite element scheme\sep pressure-robust methods 
\end{keyword}

\end{frontmatter}

\section{Introduction}

Let $\Omega\subset\mathbb{R}^d$ be a bounded Lipschitz domain with $d\in\{2,3\}$. Let $\bm{u}: \Omega\rightarrow\mathbb{R}^d$ be the velocity field of a fluid occupying $\Omega$ and $p:\Omega\rightarrow\mathbb{R}$ denote its kinematic pressure. The dynamics of the incompressible fluid within  $\Omega$ subject to the loads $\bm{f}\in L^2\big(0,T;[L^2(\Omega)]^d\big)$, $\bm{g}\in L^2\big(0,T;[L^2(\partial\Omega)]^d\big)$ before a time $T>0$ is governed by 
the incompressible Navier--Stokes equation   
\begin{subequations}\label{NSE}
    \begin{align}
    \bm{u}_t-\nu\Delta\bm{u}+\bm{u}\cdot\nabla\bm{u}+\nabla p&=\bm{f}\quad\text{ in }\Omega\times(0,T],\label{NSE1}\\
    \nabla\cdot\bm{u}&=0\quad\text{ in }\Omega\times(0,T],\\
    \bm{u}&=\bm{g}\quad\text{ on }\partial\Omega\times(0,T],\\
    \bm{u}(0)&=\bm{u}_0\quad\text{ in }\Omega,
\end{align}
\end{subequations}
where $\nu>0$ is the kinematic viscosity constant, and $\bm{u}_0\in [L^2(\Omega)]^d$ is the initial velocity. 

Numerical discretization of the velocity-pressure formulation \eqref{NSE} is challenging in several aspects. 
To achieve the linear numerical stability, it is essential to choose compatible approaches to discretizing the ${\it velocity\times pressure}$ pair, see, e.g., \cite{BoffiBrezziFortin2013,GiraultRaviart1986,BernardiRaugel1985,CrouzeixRaviart1973,ScottVogelius1985,ArnoldBrezziFortin1985,Zhang2005,GuzmanNeilan2014} for the construction of stable Stokes element pairs and \cite{DiPietroErn2012,CarreroCockburnSchotzau2006,GuzmanShuSequeira2017,ChenLiCorinaCimbala2020,ChenLiDrapacaCimbala2021} for Stokes discontinuous Galerkin methods. 
Second, when the incompressibility constraint $\nabla\cdot\bm{u}=0$ is violated on the discrete level, the performance of Stokes finite elements deteriorates as $\nu$ becomes small. In particular, the $H^1$ velocity error is dominated by the $L^2$ pressure error $\nu^{-1}\|p-p_h\|$. Those finite element discretizations, even though stable, are known as Stokes elements without pressure-robustness. Many classical works are devoted to alleviate or remove the drawback of popular non-divergence-free Stokes elements, see, e.g., the grad-div stabilization technique \cite{FrancaHughes1988,Olshanskii2002,OlshanskiiReusken2004}, postprocessed test functions \cite{LinkeMerdon2016,LinkeMatthiesTobiska2016,JohnLinkeMerdonNeilanRebholz2017}, and pointwise divergence-free Stokes elements \cite{ScottVogelius1985,Zhang2005,GuzmanNeilan2014,GuzmanNeilan2014b}.

Another difficulty in the numerical solution of the Navier--Stokes equation is the presence of the nonlinear convective term $\bm{u}\cdot\nabla\bm{u}$. When $\nu\ll1$, \eqref{NSE} becomes a convection-dominated nonlinear problem. For such problems, it is well known that standard discretization methods inevitably produce numerical solutions with non-physical oscillations. In computational fluid dynamics community, the streamline diffusion \cite{BrooksHughes1981} is a popular  technique    for handling convection-dominated flows. However, it is also known that the streamline diffusion schemes rely on an optimal choice of a parameter which is problem-dependent and ad-hoc application of this numerical technique may lead to over-diffused solutions. For elliptic convection-diffusion problems of the form $-\nabla\cdot(\alpha\nabla u+\beta u)=f$, the edge-averaged finite element (EAFE) \cite{XuZikatanov1999} is an alternative approach to discretizing convection-dominated equations without spurious oscillations in the numerical solutions and is a generalization of the traditional Scharfetter--Gummel scheme \cite{SchafetterGummel1969,BankCoughranCowsar1998} in multi-dimensional space. When compared with the streamline diffusion approach, the EAFE method is a provably monotone scheme satisfying a discrete maximum principle on a wide class of meshes. Recently EAFE has been generalized to higher oder nodal elements and edge and face finite elements, see \cite{BankVassiZikatanov2017,WuXu2020,2020WuZikatanov-a}.

A well-known fact is that the conforming $P_1\times P_0$ finite element approximation, where $P_k$ stands for  piecewise polynomials of degree at most $k$, to ${\it velocity\times pressure}$ pair is not Stokes stable. In this paper, we generalize the stabilized $P_1\times P_0$ element method in \cite{RodrigoHu2018} for the linear Stokes problem to the Navier--Stokes equation \eqref{NSE}. The work \cite{RodrigoHu2018} solves a modified discrete  Stokes system based on the classical Bernardi--Raugel (BR) element \cite{BernardiRaugel1985}. In the solution phase, degrees of freedom (dofs) associated with face bubbles in the BR element are removed in a way similar to static condensation. Because of the nonlinear convection $\bm{u}\cdot\nabla\bm{u}$, it is not clear whether the reduction of face bubbles in \cite{RodrigoHu2018} is applicable to the Navier--Stokes problem \eqref{NSE}. Moreover, the error analysis of the $P_1\times P_0$ scheme in \cite{RodrigoHu2018} has not been present in the literature to date. For that stabilized $P_1\times P_0$ method with slight modification, we shall prove a priori error estimates  uniform with respect to $\nu\ll1$.

In contrast to popular upwind techniques such as the streamline diffusion and upwind finite difference/discontinuous Galerkin discretization schemes, EAFE has not been applied to convection-dominated incompressible flows in the literature. A classical work  relevant to this paper is  \cite{XuYing2001}, where a priori error estimates of EAFE schemes for nonlinear hyperbolic conservation laws are presented. The EAFE bilinear form in \cite{XuZikatanov1999} is determined by nodal values of trial and test functions. As a result, a naive EAFE discretization for $-\nu\Delta\bm{u}+\bm{u}\cdot\nabla\bm{u}$ ignores all stabilizing face bubbles in the BR element and would lead to an unstable discretization for incompressible flows. On the other hand, face bubbles used in trial and test functions for $\bm{u}\cdot\nabla\bm{u}$ may yield a matrix prohibiting application of the bubble reduction technique proposed in \cite{RodrigoHu2018}. In this work, however, we successfully combine the aforementioned reduced BR element and EAFE discretization for the convective term and obtain a new stabilized finite element method for \eqref{NSE} with computational cost dependent on the number of dofs in $[P_1]^d\times P_0$ discretization (see Sections \ref{secOseen} and \ref{secNS} for details). In comparison with classical schemes, it turns out that the method proposed here is quite robust with respect to $\nu$ even when $\nu\ll 1$.
Besides the proposed stabilized scheme, we refer to 
\cite{HughesFrancaBalestra1986,DouglasWang1989} for other stabilized $P_k\times P_{k-1}$ numerical methods for Stokes/Navier--Stokes problems.

The rest of the paper is organized as follows. In Section \ref{secStokes}, we present a robust stabilized $P_1\times P_0$ finite element method for the Stokes problem and the error analysis uniform with respect to $\nu$. In Section \ref{secOseen}, we combine that scheme with EAFE to derive a robust method for the linear Oseen equation. Section \ref{secNS} is devoted to the robust stabilized-$(P_1\times P_0)$-EAFE scheme for the stationary and evolutionary Navier--Stokes equation. In Section \ref{secNE}, the proposed methods are tested in several benchmark problems in two and three spatial dimension. Possible extensions of this work are discussed in Section \ref{seccon}. 

\subsection{Notation}
Let $\mathcal{T}_h$ be a conforming and shape-regular simplicial partition of $\Omega$. Let  $\mathcal{F}_h$ denote the collection of $(d-1)$-dimensional faces in $\mathcal{T}_h,$ $\mathcal{E}_h$ the set of edges in $\mathcal{T}_h$, and $\{z_i\}_{i=1}^N$ the set of grid vertices in $\mathcal{T}_h$. In $\mathbb{R}^2$, the edge set $\mathcal{E}_h$ and face  set $\mathcal{F}_h$ coincide. Given $T\in\mathcal{T}_h,$ let $P_k(T)$ be the space of polynomials on $T$ of degree at most $k$. For each  $F\in\mathcal{F}_h,$ we use $\bm{n}_F$ to denote a unit vector normal to the face $F$.  Let $\lambda_i$ be the hat nodal  basis function at the vertex $z_i$. The face bubble  $\phi_F$ is a function supported on the union of two elements sharing $F$ as a face, that is, $\phi_F=\prod_{\substack{z_i\in\partial F\\
\text{ is a vertex}}}\lambda_i$. 

Let $|E|$, $|F|$, $|T|$ denote the length of $E\in\mathcal{E}_h$, area of $F\in\mathcal{F}_h$, volume of $T\in\mathcal{T}_h$, respectively, and $\fint\bullet  ds=\frac{1}{|E|}\int\bullet ds$ be the average along $E$. 
The mesh size of $\mathcal{T}_h$ is $h:=\max_{T\in\mathcal{T}_h}|T|^\frac{1}{d}$. We may use $C, C_1, C_2,\ldots$ to denote generic constants that are dependent only on the shape-regularity of $\mathcal{T}_h$ and $\Omega.$ The $L^2(\Omega)$ inner product, $L^2(\Omega)$ norm, and $H^k(\Omega)$ semi-norm are denoted by $(\bullet,\bullet)$, $\|\bullet\|$, and $|\bullet|_k$, respectively,   while by $|\bm{X}|$ we denote the Euclidean norm of a vector $\bm{X}$. The notation $A\simeq B$ means that there are constants $C_1$ and $C_2$, independent of mesh size, viscosity and other parameters of interest and such that $A\leq C_1B$ and $B\leq C_2A$. We also need the Sobolev space  $H(\operatorname{div})$ defined as follows:
\begin{equation*}
    H({\rm div},\Omega):=\big\{\bm{v}\in [L^2(\Omega)]^d: \nabla\cdot\bm{v}\in L^2(\Omega)\big\}.
\end{equation*}

\section{Stokes problem}\label{secStokes}
In order to study the incompressibility condition in \eqref{NSE}, we investigate the Stokes problem
\begin{subequations}\label{Stokes}
    \begin{align}
    -\nu\Delta\bm{u}+\nabla p&=\bm{f}\quad\text{ in }\Omega,\\
    \nabla\cdot\bm{u}&=0\quad\text{ in }\Omega,\\
    \bm{u}&=\bm{g}\quad\text{ on }\partial\Omega.
\end{align}
\end{subequations}
Consider the following space
\begin{align*}
    L^2_0(\Omega)&=\left\{q\in L^2(\Omega): \int_\Omega qdx=0\right\}.
\end{align*}
The variational formulation of \eqref{Stokes} is to find $\bm{u}\in [H^1(\Omega)]^d$ with $\bm{u}|_{\partial\Omega}=\bm{g}$ and $p\in L^2_0(\Omega)$ such that
\begin{equation}\label{Stokesvar}
  \begin{aligned}
  \nu (\nabla\bm{u},\nabla\bm{v}) - (\nabla\cdot\bm{v},p) &= (\bm{f},\bm{v}),&&\quad\forall \bm{v}\in [H^1_0(\Omega)]^d,\\
  (\nabla\cdot\bm{u},q) &= 0,&&\quad\forall q\in L^2_0(\Omega).
 \end{aligned}
\end{equation}
The construction of stable finite element subspaces of $[H^1_0(\Omega)]^d\times L^2_0(\Omega)$ was initiated in 1970s and is still under intensive investigation (cf.~\cite{BoffiBrezziFortin2013}). Let 
\begin{align*}
    &\bm{V}_h^l=\big\{\bm{v}_h\in [H^1_0(\Omega)]^d: \bm{v}_h|_T\in [P_1(T)]^d\quad\forall T\in\mathcal{T}_h\big\},\\
    &\bm{V}_h^b=\big\{\bm{v}_h\in [H^1_0(\Omega)]^d: \bm{v}_h|_T\in\text{span}\big\{\phi_F\bm{n}_F\big\}_{F\subset\partial T, F\in\mathcal{F}_h}\quad\forall T\in\mathcal{T}_h\big\}.
\end{align*}
The starting point of our scheme is the Bernardi--Raugel finite element space
\begin{align*}
    &\bm{V}_h:=\bm{V}^l_h\oplus\bm{V}_h^b,\\
    &Q_h:=\{q_h\in L^2_0(\Omega): q_h|_T\in P_0(T)\quad\forall T\in\mathcal{T}_h\},
\end{align*}
which, as shown in \cite{BernardiRaugel1985}, satisfies the inf-sup condition
\begin{equation}\label{infsup}
    \sup_{0\neq\bm{v}_h\in\bm{V}_h}\frac{(\nabla\cdot\bm{v}_h,q_h)}{|\bm{v}_h|_1}\geq\beta\|q_h\|,\quad\forall q_h\in Q_h,
\end{equation}
where $\beta>0$ is an absolute constant dependent on shape regularity of the mesh $\mathcal{T}_h$ and the domain $\Omega$. As a consequence of \eqref{infsup} and the Babu\v{s}ka--Brezzi theory \cite{Babuska1973,Brezzi1974}, the velocity-pressure error of the BR finite element method is first-order convergent under the norm $|\cdot|_1\times\|\cdot\|$. However, due to $\nabla\cdot\bm{V}_h\neq Q_h$, convergence rate of the $H^1$ velocity error and the velocity-pressure error may deteriorate severely when $\nu\to0,$ see \cite{JohnLinkeMerdonNeilanRebholz2017}.

\subsection{Stabilized $P_1\times P_0$ method for Stokes problems}
The approximation power of the BR element is provided by the linear space $\bm{V}_h^l$ while $\bm{V}_h^b$ serves only as a stabilizing component. A disadvantage of the classical BR element is that the number of dofs in $\bm{V}_h^b$ is much larger than the nodal element space $\bm{V}_h^l$. Recently the work \cite{RodrigoHu2018} is able to statically condense out dofs of face bubbles in $\bm{V}_h^b$ and obtain a stabilized $P_1\times P_0$ method for Stokes problems. Given $\bm{v}_h\in\bm{V}_h$, let \begin{equation*}
    \bm{v}_h=\bm{v}_h^l+\bm{v}_h^b,\quad\bm{v}_h^l\in\bm{V}_h^l,\quad\bm{v}_h^b\in\bm{V}_h^b 
\end{equation*} 
be the unique decomposition of $\bm{v}_h.$ We define a bilinear form $a^b_h: \bm{V}^b_h\times\bm{V}^b_h\rightarrow\mathbb{R}$ by 
\begin{equation*}
    a_h^b(\bm{u}_b,\bm{v}_b)=\sum_{F\in\mathcal{F}_h}u_Fv_F\big(\nabla(\phi_F\bm{n}_F),\nabla(\phi_F\bm{n}_F)\big),\quad\forall\bm{u}_b, \bm{v}_b\in\bm{V}_h^b,
\end{equation*}
where $\{u_F\}_{F\in\mathcal{F}_h}, \{v_F\}_{F\in\mathcal{F}_h}$ are coefficients in the unique representations  $\bm{u}_b=\sum_{F\in\mathcal{F}_h}u_F\phi_F\bm{n}_F$, $\bm{v}_b=\sum_{F\in\mathcal{F}_h}v_F\phi_F\bm{n}_F$.
In practice, $a_h^b$ corresponds to the diagonal of the representing matrix for the restricted form $a|_{\bm{V}_b\times\bm{V}_b}$. We then consider
the modified bilinear form $a_h: \bm{V}_h\times\bm{V}_h\rightarrow\mathbb{R}$ given by
\begin{equation}
    a_h(\bm{v}_h,\bm{w}_h):=a_h^b(\bm{v}_h^b,\bm{w}_h^b)+(\nabla\bm{v}_h^b,\nabla\bm{w}_h^l)+(\nabla\bm{v}_h^l,\nabla\bm{w}_h^b)+(\nabla\bm{v}_h^l,\nabla\bm{w}_h^l),\quad\forall\bm{v}_b, \bm{w}_b\in\bm{V}_h^b.
\end{equation}
The modified BR element method in \cite{RodrigoHu2018} is to find $(\bm{u}_h,p_h)\in\bm{V}_h\times Q_h$ such that
\begin{equation}\label{BR}
  \begin{aligned}
  \nu a_h(\bm{u}_h,\bm{v}_h) - (\nabla\cdot\bm{v}_h,p_h) &= (\bm{f},\bm{v}_h),&&\quad\forall \bm{v}_h\in\bm{V}_h,\\
  (\nabla\cdot\bm{u}_h,q_h) &= 0,&&\quad\forall q_h\in Q_h.
 \end{aligned}
\end{equation}
Due to the diagonal bilinear form $a_h^b,$
the dofs associated with faces in \eqref{BR} could be eliminated via a traditional static condensation, see \eqref{matrix}, \eqref{condensed} in Section \ref{secOseen} for more details. Therefore the algebraic solution procedure of \eqref{BR} is equivalent to solving a conforming $[P_1]^d\times P_0$ algebraic linear system.

However, both the original  \cite{BernardiRaugel1985} and modified \eqref{BR} BR finite element method are not robust with respect to exceedingly small $\nu\ll1$. For Stokes elements using  discontinuous  pressures, the works \cite{Linke2014,LinkeMatthiesTobiska2016} obtain pressure-robust methods by interpolating certain test functions into an $H({\rm div})$ finite element space, e.g., the Raviart--Thomas and Brezzi--Douglas--Marini (BDM) spaces (cf.~\cite{BoffiBrezziFortin2013,RaviartThomas1977,BrezziDouglasMarini1985}). Let 
$\Pi_h$ be the canonical interpolation onto the linear BDM space
\begin{equation*}
    \bm{V}_h^{\rm BDM}:=\big\{\bm{v}_h\in H({\rm div},\Omega): \bm{v}_h|_T\in [P_1(T)]^d~\forall T\in\mathcal{T}_h\big\}.
\end{equation*}
Let $P_h$ denote the $L^2$ projection onto the space of piecewise constant functions. It is well known that
\begin{equation}\label{Pihcom}
    \nabla\cdot\Pi_h\bm{v}=P_h\nabla\cdot\bm{v},\quad\forall\bm{v}\in [H^1(\Omega)]^d.
\end{equation}
Following the idea in \cite{Linke2014,LinkeMatthiesTobiska2016}, we modify the right hand side of \eqref{BR} and seek $\bm{u}_h\in\bm{V}_h$, $p_h\in Q_h$ satisfying
\begin{equation}\label{StokesBRbdm}
  \begin{aligned}
  \nu a_h(\bm{u}_h,\bm{v}_h) - (\nabla\cdot\bm{v}_h,p_h) &= (\bm{f},\Pi_h\bm{v}_h),&&\quad\forall \bm{v}_h\in\bm{V}_h,\\
  (\nabla\cdot\bm{u}_h,q_h) &= 0,&&\quad\forall q_h\in Q_h.
 \end{aligned}
\end{equation}
We shall show that \eqref{StokesBRbdm} is uniformly convergent with respect to $\nu\ll1$.
\begin{remark}
Since $\Pi_h$ preserves conforming piecewise linear functions, we have $\Pi_h\bm{v}_h=\bm{v}_h^l+\Pi_h\bm{v}_h^b$ for $\bm{v}_h\in\bm{V}_h$. Consider the following lowest-order Raviart--Thomas space
\begin{equation*}
    \bm{V}_h^{\rm RT}:=\big\{\bm{v}_h\in H({\rm div},\Omega): \bm{v}_h|_T\in P_0(T)\bm{x}+[P_0(T)]^d~\forall T\in\mathcal{T}_h\big\},
\end{equation*}
where $\bm{x}=(x_1,\ldots,x_d)^\top$ is the coordinate vector field. Let $\bm{\phi}_F^{\rm RT}\in\bm{V}_h^{\rm RT}$ be the canonical face basis function of $\bm{V}_h^{\rm RT}$  such that $\int_F\bm{\phi}_F^{\rm RT}\cdot\bm{n}_{F^\prime}dS=\delta_{F,F^\prime}$ for all $F, F^\prime\in\mathcal{F}_h$ with $\delta_{F,F^\prime} $ being the Kronecker delta symbol. Direct calculation shows that
\begin{equation}\label{piphiene}
    \Pi_h(\phi_F\bm{n}_F)=\left\{\begin{aligned}
    \frac{|F|}{6}\bm{\phi}_F^{\rm RT}\quad in\quad \mathbb{R}^2,\\
    \frac{|F|}{15}\bm{\phi}_F^{\rm RT}\quad in\quad \mathbb{R}^3.
    \end{aligned}\right.
\end{equation}
\end{remark}

To illustrate the effectiveness of the new scheme \eqref{StokesBRbdm}, we check the performance of \eqref{StokesBRbdm} and \eqref{BR} applied to the Stokes problem \eqref{Stokes} with $\nu=10^{-3}$ and the exact solution \begin{align*}
    &\bm{u}(\bm{x})=\big(-\sin(\pi x_1)^2\sin(2\pi x_2), \sin(2\pi x_1)\sin(\pi x_2)^2\big)^\top,\\
    &p(\bm{x})=\exp(x_1+x_2)-(\exp(1)-1)^2.
\end{align*}
We set $\Omega=[0,1]\times[0,1]$ to be the unit square and consider the homogeneous Dirichlet boundary condition $\bm{g}=\bm{0}$. The domain $\Omega$ is partitioned into a  $16\times16$ uniform grid of right triangles. The numerical solutions of \eqref{BR} and \eqref{StokesBRbdm} are visualized using the MATLAB function \texttt{quiver} in Fig.~\ref{Stokesplot}. On such a coarse mesh, the velocity by \eqref{StokesBRbdm} is observed to be a good approximation to the sinusoidal solution $\bm{u}$ while the qualitative behavior of the velocity by \eqref{BR} is completely misleading for the small viscosity $\nu$.
\begin{figure}[!hptb]
\centering
\begin{tabular}[c]{cccc}%
  \subfigure[Plot of $\bm{u}_h$ from \eqref{BR}]{\includegraphics[width=6cm,height=6cm]{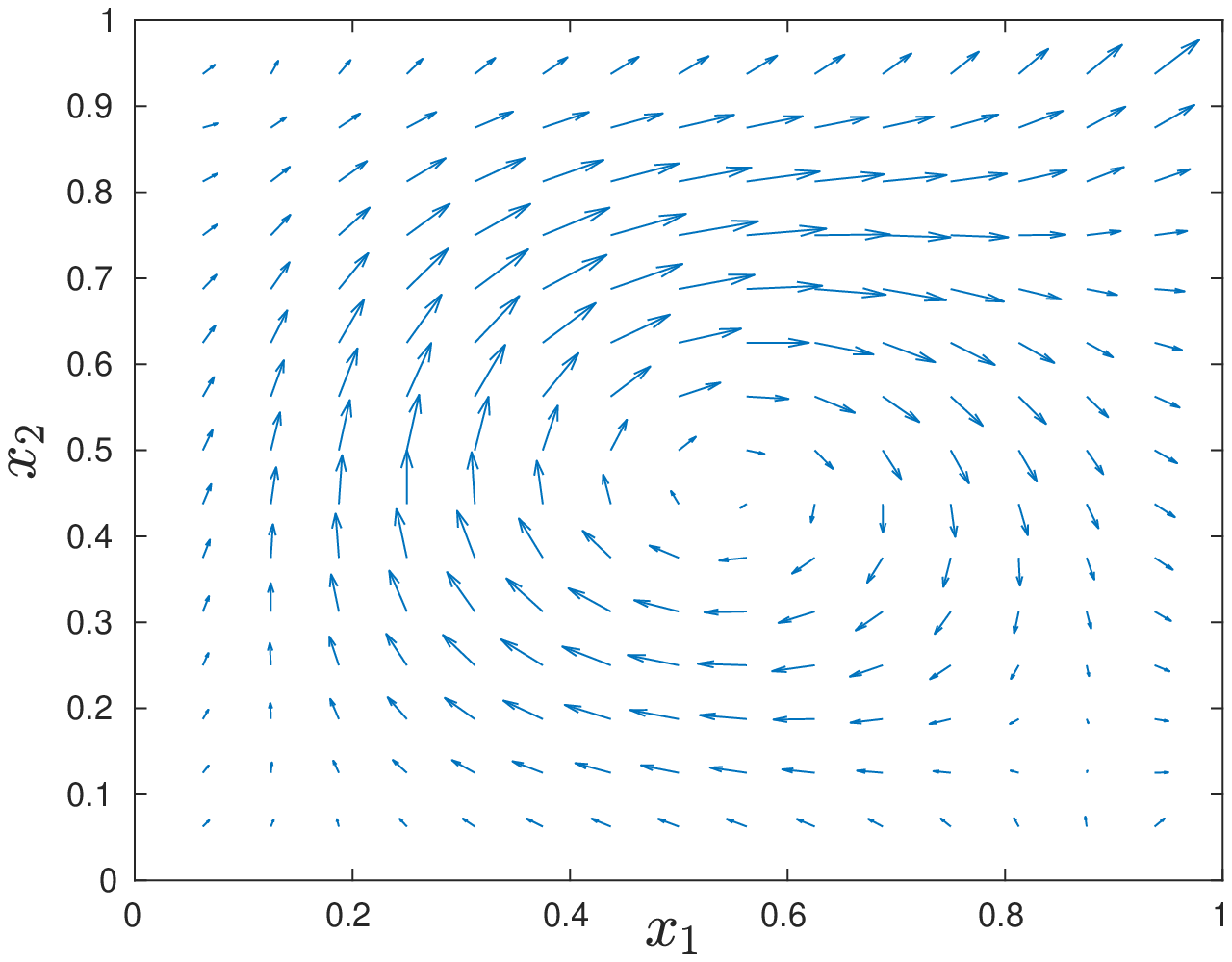}}
  \subfigure[Plot of $\bm{u}_h$ from \eqref{StokesBRbdm}]{\includegraphics[width=6cm,height=6cm]{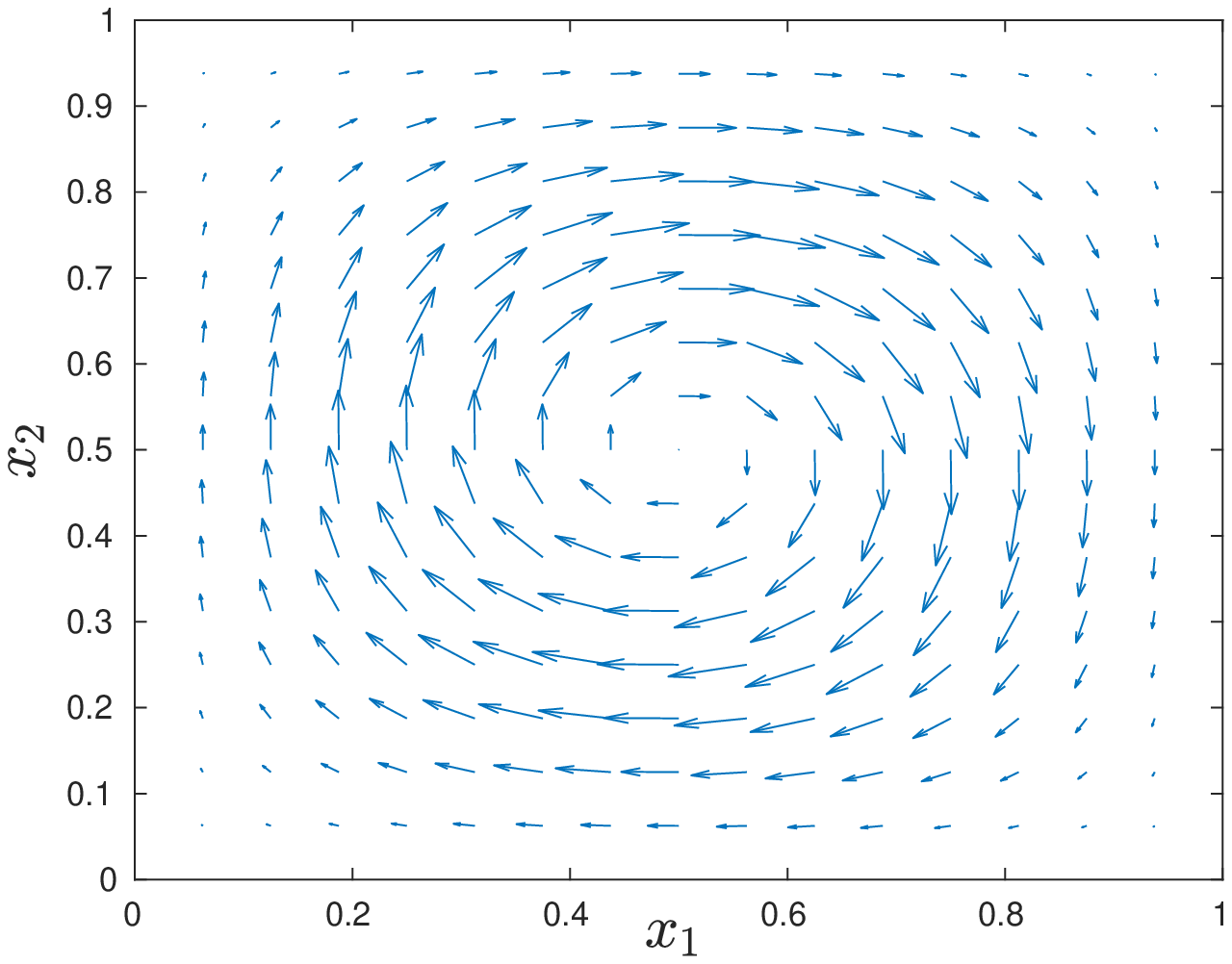}}
\end{tabular}
\caption{Numerical velocity fields  for \eqref{Stokes} with $\nu=10^{-3}$ on a $16\times16$ uniform triangulation of $\Omega=[0,1]^2$.}
\label{Stokesplot}
\end{figure}

\subsection{Convergence analysis}
As we have pointed out earlier, the a priori error analysis of the stabilized $P_1\times P_0$ method \eqref{BR} has not been established in the literature. In this subsection, we go one step further and present a new error estimate for the modified scheme \eqref{StokesBRbdm} that is robust with respect to $\nu$, when $\nu\ll 1$. Our approach is to compare the error in the numerical approximation given by \eqref{StokesBRbdm} with the  error in the following scheme: Find $(\tilde{\bm{u}}_h,\tilde{p}_h)\in\bm{V}_h\times Q_h$ such that
\begin{equation}\label{BRbdm0}
  \begin{aligned}
  \nu (\nabla\tilde{\bm{u}}_h,\nabla\bm{v}_h) - (\nabla\cdot\bm{v}_h,\tilde{p}_h) &= (\bm{f},\Pi_h\bm{v}_h),&&\quad\forall \bm{v}_h\in\bm{V}_h,\\
  (\nabla\cdot\tilde{\bm{u}}_h,q_h) &= 0,&&\quad\forall q_h\in Q_h.
 \end{aligned}
\end{equation}
Following the analysis in \cite{LinkeMatthiesTobiska2016}, it is straightforward to show that
\begin{subequations}\label{BR0est}
\begin{align}
   &|\bm{u}-\tilde{\bm{u}}_h|_1\leq Ch|\bm{u}|_2,\label{BR0u}\\  &\|p-\tilde{p}_h\|\leq Ch\big(\nu|\bm{u}|_2+|p|_1\big).\label{BR0p}
\end{align}
  \end{subequations}

A local homogeneity argument implies
\begin{equation}
    a_h^b(\bm{v}_b,\bm{w}_b)\simeq(\nabla\bm{v}_b,\nabla\bm{w}_b),\quad\forall \bm{v}_b, \bm{w}_b\in\bm{V}_h^b,
\end{equation}
i.e., $(\nabla\bullet,\nabla\bullet)|_{\bm{V}^b_h\times\bm{V}_h^b}$ is spectrally equivalent to $a_h^b.$ As a result, the modified bilinear form $a_h(\bullet,\bullet)$ is coercive 
\begin{equation}\label{ahnorm}
    |\bm{v}|^2_{1,h}:=a_h(\bm{v}_h,\bm{v}_h)\simeq|\bm{v}_h|^2_1,\quad\forall\bm{v}_h\in\bm{V}_h.
\end{equation}

Let $\bm{V}_h^\prime$, $Q_h^\prime$ be the dual space of $\bm{V}_h$, $Q_h$, respectively. Given arbitrary functionals $\bm{F}\in\bm{V}^\prime_h$, $g\in Q^\prime_h$, we consider the following problem: Find  $(\bm{v}^{F,g}_h,q^{F,g}_h)\in\bm{V}_h\times Q_h$ satisfying 
\begin{equation*}
  \begin{aligned}
  \nu a_h(\bm{v}^{F,g}_h,\bm{w}_h) - (\nabla\cdot\bm{w}_h,q^{F,g}_h) &=\bm{F}(\bm{w}_h),&&\quad\forall \bm{w}_h\in\bm{V}_h,\\
  (\nabla\cdot\bm{v}^{F,g}_h,r_h) &= g(r_h),&&\quad\forall r_h\in Q_h.
 \end{aligned}
\end{equation*}
Let $\bm{V}_h$ be equipped with the norm $\nu^\frac{1}{2}|\bullet|_1$ and $Q_h$ use the norm $\nu^{-\frac{1}{2}}\|\bullet\|$. Using the inf-sup condition \eqref{infsup}, the coercivity of $a_h$, and the classical Babu\v{s}ka--Brezzi theory (cf.~\cite{Brezzi1974,XuZikatanov2003}), we obtain the following stability estimate
\begin{equation}\label{stability}
    \nu^\frac{1}{2}|\bm{v}^{F,g}_h|_{1,h}+\nu^{-\frac{1}{2}}\|q^{F,g}_h\|\leq\mathcal{K}\left(\beta^{-1}\right)\big(\sup_{\nu^\frac{1}{2}|\bm{w}_h|_1=1}\bm{F}(\bm{w}_h)+\sup_{\nu^{-\frac{1}{2}}\|r_h\|=1}g(r_h)\big),
\end{equation}
where $\mathcal{K}$ is a fixed increasing function. Now we are in a position to present a robust error estimate of \eqref{StokesBRbdm}.
\begin{theorem}\label{BRest}
For $(\bm{u}_h,p_h)$ given in \eqref{StokesBRbdm}, there exist absolute constants $C_u$, $C_p$ independent of $\nu$ and $h$, and such that
    \begin{subequations}
\begin{align}
   &|\bm{u}-\bm{u}_h|_1\leq C_uh|\bm{u}|_2,\label{urobust}\\  &\|p-p_h\|\leq C_ph\big(\nu|\bm{u}|_2+|p|_1\big).\label{probust}
\end{align}
  \end{subequations}
\end{theorem}

\begin{proof}
Consider the space of weakly divergence-free functions \begin{equation*}
    \bm{W}_h:=\big\{\bm{v}_h\in\bm{V}_h: (\nabla\cdot\bm{v}_h,q_h)=0~\forall q_h\in Q_h\big\}.
\end{equation*}
By the definitions of
$(\bm{u}_h,p_h)$ in \eqref{StokesBRbdm} and $(\tilde {\bm{u}}_h,\tilde p_h)$ in \eqref{BRbdm0}, we have for all $\bm{v}\in\bm{V}_h$ and $q\in Q_h$:
\begin{subequations}
  \begin{align}
  a_h(\bm{u}_h,\bm{v}_h)& =
    (\nabla\tilde{\bm{u}}_h,\nabla\bm{v}_h)=0,\quad\forall \bm{v}_h\in \bm{W}_h,\label{ortho1}\\
    \nu a_h(\bm{u}_h,\bm{v}_h) -(\nabla\cdot\bm{v}_h,p_h) & =
    \nu(\nabla\tilde{\bm{u}}_h,\nabla\bm{v}_h) - (\nabla\cdot\bm{v}_h,\tilde{p}_h),\quad\forall \bm{v}_h\in \bm{V}_h,\label{ortho2}\\
  (\nabla\cdot\bm{u}_h,q_h) & = (\nabla\cdot\tilde{\bm{u}}_h,q_h)=0,\quad\forall q_h\in Q_h.\label{ortho3}
 \end{align}
\end{subequations}
It then follows from \eqref{ortho1} and $a_h(\Pi_h\bm{u},\bm{v}_h)=(\nabla\Pi_h\bm{u},\nabla\bm{v}_h)$ that
\begin{equation}\label{uhpihu}
    a_h(\bm{u}_h-\Pi_h\bm{u},\bm{v}_h)=
    \big(\nabla(\tilde{\bm{u}}_h-\Pi_h\bm{u}),\nabla\bm{v}_h\big),\quad\forall \bm{v}_h\in \bm{W}_h.
\end{equation}
Combining \eqref{Pihcom} and $\nabla\cdot\bm{u}=0$, we conclude that $\nabla\cdot\Pi_h\bm{u}=0$ and thus $\Pi_h\bm{u}\in\bm{W}_h$. Therefore
taking $\bm{v}_h=\bm{u}_h-\Pi_h\bm{u}\in\bm{W}_h$ in \eqref{uhpihu} and using \eqref{ahnorm} and the Cauchy--Schwarz inequality lead to
\begin{equation*}
    |\bm{u}_h-\Pi_h\bm{u}|^2_1\simeq a_h(\bm{v}_h,\bm{v}_h)\leq|\tilde{\bm{u}}_h-\Pi_h\bm{u}|_1|\bm{v}_h|_1.
\end{equation*}
As a consequence, it holds that
\begin{equation}\label{term2}
\begin{aligned}
    |\bm{u}_h-\Pi_h\bm{u}|_1&\leq C|\tilde{\bm{u}}_h-\Pi_h\bm{u}|_1\\
    &\leq C|\bm{u}-\tilde{\bm{u}}_h|_1+C|\bm{u}-\Pi_h\bm{u}|_1.
\end{aligned}
\end{equation}
Using \eqref{term2} and a triangle inequality, we have
\begin{equation}\label{lastu}
\begin{aligned}
    |\bm{u}-\bm{u}_h|_1&\leq|\bm{u}-\Pi_h\bm{u}|_1+|\bm{u}_h-\Pi_h\bm{u}|_1\\
    &\leq C|\bm{u}-\Pi_h\bm{u}|_1+C|\bm{u}-\tilde{\bm{u}}_h|_1.
\end{aligned}
\end{equation}
Combining \eqref{lastu} with \eqref{BR0u} and the interpolation error estimate of $\Pi_h$
\begin{equation*}
    |\bm{u}-\Pi_h\bm{u}|_1\leq Ch|\bm{u}|_2
\end{equation*}
finishes the proof of \eqref{urobust}.

Next, we pick arbitrary $\bm{w}_l\in\bm{V}_h^l$ and $r_h \in Q_h$
and use \eqref{ortho2}, \eqref{ortho3} to obtain that
\begin{subequations}\label{compareup}
  \begin{align}
    \nu a_h(\bm{u}_h-\bm{w}_l,\bm{v}_h) -(\nabla\cdot\bm{v}_h,p_h-r_h) & =
    \nu \big(\nabla(\tilde{\bm{u}}_h-\bm{w}_l),\nabla\bm{v}_h\big) - (\nabla\cdot\bm{v}_h,\tilde{p}_h-r_h),\quad\forall \bm{v}_h\in \bm{V}_h,\\
  (\nabla\cdot(\bm{u}_h-\bm{w}_l),q_h) & = (\nabla\cdot(\tilde{\bm{u}}_h-\bm{w}_l),q_h)=0,\quad\forall q_h\in Q_h.
 \end{align}
\end{subequations}
A combination of \eqref{compareup} and the stability estimate \eqref{stability} then implies that for all $\bm{w}_l\in \bm{V}_l$, and $r_h\in Q_h$,
\begin{equation}\label{errprh}
\begin{aligned}
  \nu^\frac{1}{2}|\bm{u}_h-\bm{w}_l|_{1,h}+\nu^{-\frac{1}{2}}\|p_h-r_h\|&\leq\mathcal{K}(\beta^{-1})
  \big(\nu^\frac{1}{2}|\tilde{\bm{u}}_h-\bm{w}_l|_{1,h}+\nu^{-\frac{1}{2}}\|\tilde{p}_h-r_h\|\big)\\
  &\simeq\mathcal{K}(\beta^{-1})
  \big(
  \nu^\frac{1}{2}|\tilde{\bm{u}}_h-\bm{w}_l|_1+\nu^{-\frac{1}{2}}\|\tilde{p}_h-r_h\|\big)\\
  &\leq C\big(\nu^\frac{1}{2}|\bm{u}-\tilde{\bm{u}}_h|_1+\nu^\frac{1}{2}|\bm{u}-\bm{w}_l|_1 + \nu^{-\frac{1}{2}}\|p-\tilde{p}_h\|  +\nu^{-\frac{1}{2}}\|p-r_h\|\big).
\end{aligned}
\end{equation}
Using \eqref{errprh} and the triangle inequality, we have,
\begin{align*}
  \|p-p_h\|
  &\leq\|p-r_h\|+\|p_h-r_h\|\\
  &\leq C\big(\|p-r_h\|+\nu|\bm{u}-\tilde{\bm{u}}_h|_1+\nu|\bm{u}-\bm{w}_l|_1 + \|p-\tilde{p}_h\|\big).
\end{align*}
Finally taking the infimum with respect to $(\bm{w}_l,r_h)\in \bm{V}_h^l\times Q_h$ in the above inequality and using \eqref{urobust} concludes the proof of \eqref{probust}.
\qed\end{proof}

\section{Stationary Oseen equation}\label{secOseen}
We now move on to the discretization of the stationary Navier--Stokes equation
\begin{subequations}\label{SNSE}
    \begin{align}
    -\nu\Delta \bm{u}+\bm{u}\cdot\nabla\bm{u}+\nabla p&=\bm{f}\quad\text{ in }\Omega,\label{SNSE1}\\
    \nabla\cdot\bm{u}&=0\quad\text{ in }\Omega,\\
    \bm{u}&=\bm{g}\quad\text{ on }\partial\Omega,
\end{align}
\end{subequations}
with a potentially very small viscosity $\nu\ll 1$.
As a first step, we consider the linear Oseen equation, which naturally occurs when \eqref{SNSE} is linearized by a fixed point iteration: 
\begin{subequations}\label{Oseen}
    \begin{align}
    -\nu\Delta \bm{u}+\bm{b}\cdot\nabla\bm{u}+\nabla p&=\bm{f}\quad\text{ in }\Omega,\label{Oseen1}\\
    \nabla\cdot\bm{u}&=0\quad\text{ in }\Omega,\\
    \bm{u}&=\bm{g}\quad\text{ on }\partial\Omega,
\end{align}
\end{subequations}
 Here, $\bm{b}: \Omega\rightarrow\mathbb{R}^d$ is a given convective field. For a reason which will become apparent later, we add a viscous term $-\varepsilon\Delta\bm{u}$ with vanishingly small $\varepsilon$,  $0<\varepsilon\ll1$ in \eqref{Oseen1} and obtain the modified momentum equation
 \begin{equation}\label{mmomentum}
     -\nu\Delta \bm{u}-\varepsilon\Delta \bm{u}+\bm{b}\cdot\nabla\bm{u}+\nabla p=\bm{f}.
 \end{equation}
Given a vector-valued $\bm{v}=(v_1,\ldots,v_d)^\top$, we consider the second order flux tensor \begin{equation}\label{EAFEflux}
 \mathbf{J}(\bm{v}):=\varepsilon\nabla\bm{v}+\bm{v}\otimes\bm{b}.
    \end{equation}
The $i$-th row of $\mathbf{J}(\bm{v})=(\bm{J}(v_1)^\top,\ldots,\bm{J}(v_d)^\top)^\top$ then is
the standard flux
\[
\bm{J}(v_i):=\varepsilon\nabla v_i+\bm{b}v_i
\]
used in the classical EAFE discretization \cite{XuZikatanov1999}. 

For each  $E\in\mathcal{E}_h,$ let $\bm{t}_E$ be a fixed unit vector tangent to the edge $E$, and let $\psi_E\in C^1(E)$ satisfy 
    \begin{equation*}
        \partial_{\bm{t}_E}\psi_E=\varepsilon^{-1}\bm{b}\cdot\bm{t}_E,
    \end{equation*}
    where $\partial_{\bm{t}_E}$ is the directional derivative along $\bm{t}_E.$ We remark that $\psi_E$ is determined up to a constant.
The following lemma could be found in \cite{XuZikatanov1999}.
\begin{lemma}\label{EAFElemma}
Let $\bm{\tau}_E=|E|\bm{t}_E$ be a non-unit vector tangent to $E\in\mathcal{E}_h$. We have
\begin{equation}\label{EAFE2}
    \delta_E(e^{\psi_E}v)=\fint_E\varepsilon^{-1}e^{\psi_E}\bm{J}(\bm{v})\cdot\bm{\tau}_E ds.
\end{equation}
For a piecewise constant $\bm{J}_h$ vector field and a continuous and piecewise affine $w_h$ with respect to $\mathcal{T}_h$, we have
\begin{equation}\label{EAFE1}
    (\nabla w_h,\bm{J}_h)=-\sum_{E\in\mathcal{E}_h}a_E\delta_Ew_h\bm{J}_h\cdot\bm{\tau}_E,
\end{equation}
where $\delta_E \chi= \chi(x_j)-\chi(x_i)$ with $E$ being the edge having endpoints $x_i$, $x_j$, $\bm{\tau}_E$ points from $x_i$ to $x_j$, and $$a_E=a_{ij}:=\int_\Omega\nabla\lambda_i\cdot\nabla\lambda_jdx$$ is the $(i,j)$-entry of $P_1$ element stiffness matrix.
\end{lemma}

For $\bm{v}=(v_1,\ldots,v_d)^\top\in\bm{V}_h^l$, let $\mathbf{J}_h=(\bm{J}_h^{1\top},\ldots,\bm{J}_h^{d\top})^\top=P_h\mathbf{J}(\bm{v})$ be the piecewise constant $L^2$ projection of $\mathbf{J}(\bm{v})$ and thus $\bm{J}_h^i\approx\bm{J}(v_i)$.
It follows from \eqref{EAFE2} with $v=v_i$ that
\begin{equation*}
    \bm{J}_h^i\cdot\bm{\tau}_E\approx\varepsilon \left(\fint_Ee^{\psi_E}ds\right)^{-1}\delta_E(e^{\psi_E}v_i).
\end{equation*}
Using this fact, integration by parts and \eqref{EAFE1}, for a piecewise linear finite element  approximation $\bm{u}_h^l=(\bm{u}_{h,1}^l,\ldots,\bm{u}_{h,d}^l)^\top$ to $\bm{u}$, we have
\begin{equation}
    \begin{aligned}
    &(-\varepsilon\Delta \bm{u}+\bm{b}\cdot\nabla\bm{u},\bm{v})=(\nabla\bm{u},\varepsilon\nabla\bm{v}+\bm{v}\otimes\bm{b})\\
    &=(\nabla\bm{u},\mathbf{J}(\bm{v}))\approx(\nabla\bm{u}_h^l,\mathbf{J}_h)=-\sum_{i=1}^d\sum_{E\in\mathcal{E}_h}a_E\delta_Eu_{h,i}^l\bm{J}^i_h\cdot\bm{\tau}_E\\
    &\approx-\sum_{i=1}^d\sum_{E\in\mathcal{E}_h}\varepsilon a_E\left(\fint_Ee^{\psi_E}ds\right)^{-1}\delta_Eu_{h,i}^l\delta_E(e^{\psi_E}v_i)\\
    &=-\sum_{E\in\mathcal{E}_h}\varepsilon a_E\left(\fint_Ee^{\psi_E}ds\right)^{-1}\delta_E\bm{u}_h^l\cdot\delta_E(e^{\psi_E}\bm{v}).
    \end{aligned}
\end{equation}
As a consequence, we define the EAFE bilinear form $b^{\rm EAFE}_h: \bm{V}_h^l\times \bm{V}_h^l\rightarrow\mathbb{R}$ by
\begin{equation*}
        b^{\rm EAFE}_h(\bm{v}^l_h,\bm{w}^l_h;\bm{b}):=-\sum_{E\in\mathcal{E}_h}\varepsilon a_E\left(\fint_Ee^{\psi_E}ds\right)^{-1}\delta_E\bm{v}^l_h\cdot\delta_E(e^{\psi_E}\bm{w}^l_h),\quad\forall\bm{v}_h^l, \bm{w}_h^l\in\bm{V}_h^l,
\end{equation*}
which is an variational form for discretizing $-\varepsilon\Delta \bm{u}+\bm{b}\cdot\nabla\bm{u}\approx\bm{b}\cdot\nabla\bm{u}$. We remark that the value of $b^{\rm EAFE}_h$ is not affected by the generic additive constant in $\psi_E$. 

\begin{remark}
The original EAFE method \cite{XuZikatanov1999} is designed for the convection-diffusion equation $-\nabla\cdot(\alpha\nabla u+\beta u)=f$ in conservative form. On the other hand, the EAFE bilinear form $b_h^{\rm EAFE}$ works for $-\nabla\cdot(\alpha\nabla u)+\beta\cdot\nabla u=f$ in convective form. In fact, $b_h^{\rm EAFE}$ depends on test functions weighted by edge-wise exponential functions while the classical EAFE scheme \cite{XuZikatanov1999} makes use of an exponentially averaged trial function.
\end{remark}

Using the modified equation \eqref{mmomentum},  and the approximations $a(\bm{u},\bm{v}_h)\approx a_h(\bm{u}_h,\bm{v}_h)$ and $(\bm{b}\cdot\nabla\bm{u},\bm{v}_h^l)\approx b^{\rm EAFE}_h(\bm{u}_h^l,\bm{v}^l_h;\bm{b})$, we obtain a stabilized  $P_1\times P_0$ EAFE scheme for the problem \eqref{Oseen}: Find $(\bm{u}_h,p_h)\in\bm{V}_h\times Q_h$ such that
\begin{equation}\label{BREAFE}
  \begin{aligned}
  \nu a_h(\bm{u}_h,\bm{v}_h) + b^{\rm EAFE}_h(\bm{u}^l_h,\bm{v}^l_h;\bm{b})- (\nabla\cdot\bm{v}_h,p_h) &= (\bm{f},\Pi_h\bm{v}_h),&&\quad\forall \bm{v}_h\in\bm{V}_h,\\
  (\nabla\cdot\bm{u}_h,q_h) &= 0,&&\quad\forall q_h\in Q_h.
 \end{aligned}
\end{equation}
Here $b^{\rm EAFE}_h(\bm{v}_h,\bm{w}_h)$ is also defined for any continuous functions $\bm{v}_h$ and $\bm{w}_h$. However, it is easy to see that the value of $b^{\rm EAFE}_h$  is determined by nodal values of trial and test functions, i.e., 
$b^{\rm EAFE}_h(\bm{v}_h,\bm{w}_h;\bm{b})=b^{\rm EAFE}_h(\bm{v}^l_h,\bm{w}^l_h;\bm{b})$ for $\bm{v}_h, \bm{w}_h\in\bm{V}_h$. 
Thus, the stabilizing face bubble functions in the BR element do not enter the discretized EAFE convective term.

Unfortunately, as numerical experiments show, the EAFE discretization given above for $\bm{b}\cdot\nabla\bm{u}$ does not work well. We now turn  to the pressure-robust finite elements using discontinuous pressures \cite{Linke2014,LinkeMatthiesTobiska2016,LinkeMerdon2016}. Those works  apply Raviart--Thomas or BDM interpolation to the test functions for the convective term $\bm{b}\cdot\nabla\bm{u}$. Similarly, we add a stabilization term with  postprocessed bubble test functions to $b_h^{\rm EAFE}$  and obtain the following discrete convective form
    \begin{equation*}
    \begin{aligned}
            &b_h(\bm{v}_h,\bm{w}_h;\bm{b})=b_h(\bm{v}^l_h,\bm{w}_h;\bm{b})\\
            &\quad:=b_h^{\rm EAFE}(\bm{v}^l_h,\bm{w}^l_h;\bm{b})+(\bm{b}\cdot\nabla\bm{v}_h^l,\Pi_h\bm{w}_h^b),\quad\forall\bm{v}_h, \bm{w}_h\in\bm{V}_h.
    \end{aligned}
    \end{equation*}
    The resulting scheme for \eqref{Oseen} is: Find $(\bm{u}_h,p_h)\in\bm{Q}_h\times V_h$ such that
\begin{equation}\label{OseenBRbdmEAFE}
  \begin{aligned}
  \nu a_h(\bm{u}_h,\bm{v}_h) + b_h(\bm{u}^l_h,\bm{v}_h;\bm{b})- (\nabla\cdot\bm{v}_h,p_h) &= (\bm{f},\Pi_h\bm{v}_h),\quad\forall \bm{v}_h\in\bm{V}_h,\\
  (\nabla\cdot\bm{u}_h,q_h) &= 0,\quad\forall q_h\in Q_h.
 \end{aligned}
\end{equation}
In \eqref{OseenBRbdmEAFE}, let $\mathbf{A}_{bb}$, $\mathbf{A}_{ll}$, $\mathbf{A}_{bl}$, $\mathbf{A}_{lb}$, $\mathbf{A}_{bp}$, $\mathbf{A}_{lp}$ be the  matrices representing the bilinear forms $\nu a_h(\bm{u}_h^b,\bm{v}_h^b)$, $\nu a_h(\bm{u}_h^l,\bm{v}_h^l)+b_h(\bm{u}_h^l,\bm{v}_h^l)$, $\nu a_h(\bm{u}_h^l,\bm{v}_h^b)+(\bm{b}\cdot\nabla\bm{u}_h^l,\Pi_h\bm{v}_h^b)$, $\nu a_h(\bm{u}_h^b,\bm{v}_h^l)$,  $-(\nabla\cdot\bm{v}_h^b,p_h)$, $-(\nabla\cdot\bm{v}_h^l,p_h)$, respectively. Furthermore, let $\bm{U}_b$, $\bm{U}_l$, $\bm{P}$, $\bm{F}_b$, $\bm{F}_l$ be the vector representation for $\bm{u}_h^b$, $\bm{u}_h^l$, $p_h$,  $(\bm{f},\Pi_h\bm{v}_h^b)$ and $(\bm{f},\bm{v}_h^l)$ (here $\Pi_h\bm{v}_h^l=\bm{v}_h^l$), respectively. The algebraic linear system for \eqref{OseenBRbdmEAFE} then reads
\begin{equation}\label{matrix}
    \begin{pmatrix}
    \mathbf{A}_{bb}&\mathbf{A}_{bl}&\mathbf{A}_{bp}\\
    \mathbf{A}_{lb}&\mathbf{A}_{ll}&\mathbf{A}_{lp}\\
    \mathbf{A}_{bp}^\top&\mathbf{A}_{lp}^\top&\mathbf{O}
    \end{pmatrix}\begin{pmatrix}
    \bm{U}_b\\\bm{U}_l\\\bm{P}
    \end{pmatrix}=\begin{pmatrix}
    \bm{F}_b\\\bm{F}_l\\\bm{0}
    \end{pmatrix}.
\end{equation}
Using the block Gaussian elimination, i.e. the static condensation, we obtain the reduced system
\begin{equation}\label{condensed}
    \begin{pmatrix}
    \mathbf{A}_{ll}-\mathbf{A}_{lb}\mathbf{A}^{-1}_{bb}\mathbf{A}_{bl}&\mathbf{A}_{lp}-\mathbf{A}_{lb}\mathbf{A}^{-1}_{bb}\mathbf{A}_{bp}\\
    \mathbf{A}_{lp}^\top-\mathbf{A}_{bp}^\top\mathbf{A}^{-1}_{bb}\mathbf{A}_{bl}&-\mathbf{A}_{bp}^\top\mathbf{A}^{-1}_{bb}\mathbf{A}_{bp}
    \end{pmatrix}\begin{pmatrix}
    \bm{U}_l\\\bm{P}
    \end{pmatrix}=\begin{pmatrix}
    \bm{F}_l-\mathbf{A}_{lb}\mathbf{A}_{bb}^{-1}\bm{F}_b\\-\mathbf{A}_{bp}^\top\mathbf{A}^{-1}_{bb}
    \end{pmatrix}.
\end{equation}
Let $N_v$ denote the number of interior vertices of $\mathcal{T}_h$, and $N_t$ the number of elements in $\mathcal{T}_h$.
Since $\mathbf{A}_{bb}$ is a diagonal matrix, the stiffness matrix in \eqref{condensed} is sparse of dimension  $(dN_v+N_t)\times(dN_v+N_t)$, the same as a $[P_1]^d\times P_0$ element. The vector $\bm{U}_b$ of dofs  for the bubble component  $\bm{u}^b_h$ are then recovered by $\bm{U}_b=\mathbf{A}_{bb}^{-1}(\bm{F}_b-\mathbf{A}_{bl}\bm{U}_l-\mathbf{A}_{bp}\bm{P})$.

\begin{remark}
We remark that the term $(\bm{b}\cdot\nabla\bm{u}_h^b,\Pi_h\bm{v}_h^b)$ is not added to \eqref{OseenBRbdmEAFE}. Otherwise the block $\mathbf{A}_{bb}$ is not diagonal  and the above bubble reduction technique does not hold.
\end{remark}

Let $b_{h,T}$ be the restriction of $b_h$ on $T\in\mathcal{T}_h$ and $E\in\mathcal{E}_h$ be an edge having endpoints $x_i$ and $x_j.$
When $\bm{b}$ is piecewise constant with respect to $\mathcal{T}_h$, the $(i,j)$-entry $b_{h,T}(\lambda_j,\lambda_i)$ of the local stiffness matrix for $b_{h,T}|_{\bm{V}_h^l\times\bm{V}_h^l}$ with $i\neq j$ is simple and written as
\begin{equation}\label{localbh}
    b_{h,T}(\lambda_j,\lambda_i)=\left\{
    \begin{aligned}
    &\varepsilon a_{ij}B(\varepsilon^{-1}\bm{b}|_T\cdot\bm{\tau}_E),\quad \bm{\tau}_E\text{ points from }x_i\text{ to }x_j,\\
    &\varepsilon a_{ij}B(-\varepsilon^{-1}\bm{b}|_T\cdot\bm{\tau}_E),\quad \bm{\tau}_E\text{ points from }x_j\text{ to }x_i,
    \end{aligned}\right.
\end{equation}
where $B$ is the Bernoulli function \begin{equation}
   B(s)=\left\{\begin{aligned}
   &\frac{s}{e^s-1},\quad &&s\neq0,\\
   &1,\quad &&s=0.
   \end{aligned}\right.
\end{equation}
For $i=j$, we have $b_{h,T}(\lambda_i,\lambda_i)=-\sum_{j\neq i}b_{h,T}(\lambda_i,\lambda_j).$
Therefore computing $b_h|_{\bm{V}_h^l\times\bm{V}_h^l}$ is no more complicated than assembling the stiffness matrix $(a_{ij})_{1\leq i,j\leq N}$ of the $P_1$ finite element method for Poisson's equation.

\begin{figure}[ptb]
\centering
\begin{tabular}[c]{cccc}%
  \subfigure[Plot of $|\bm{u}_h|$ from \eqref{OseenBRbdmEAFE}]{\includegraphics[width=6cm,height=6cm]{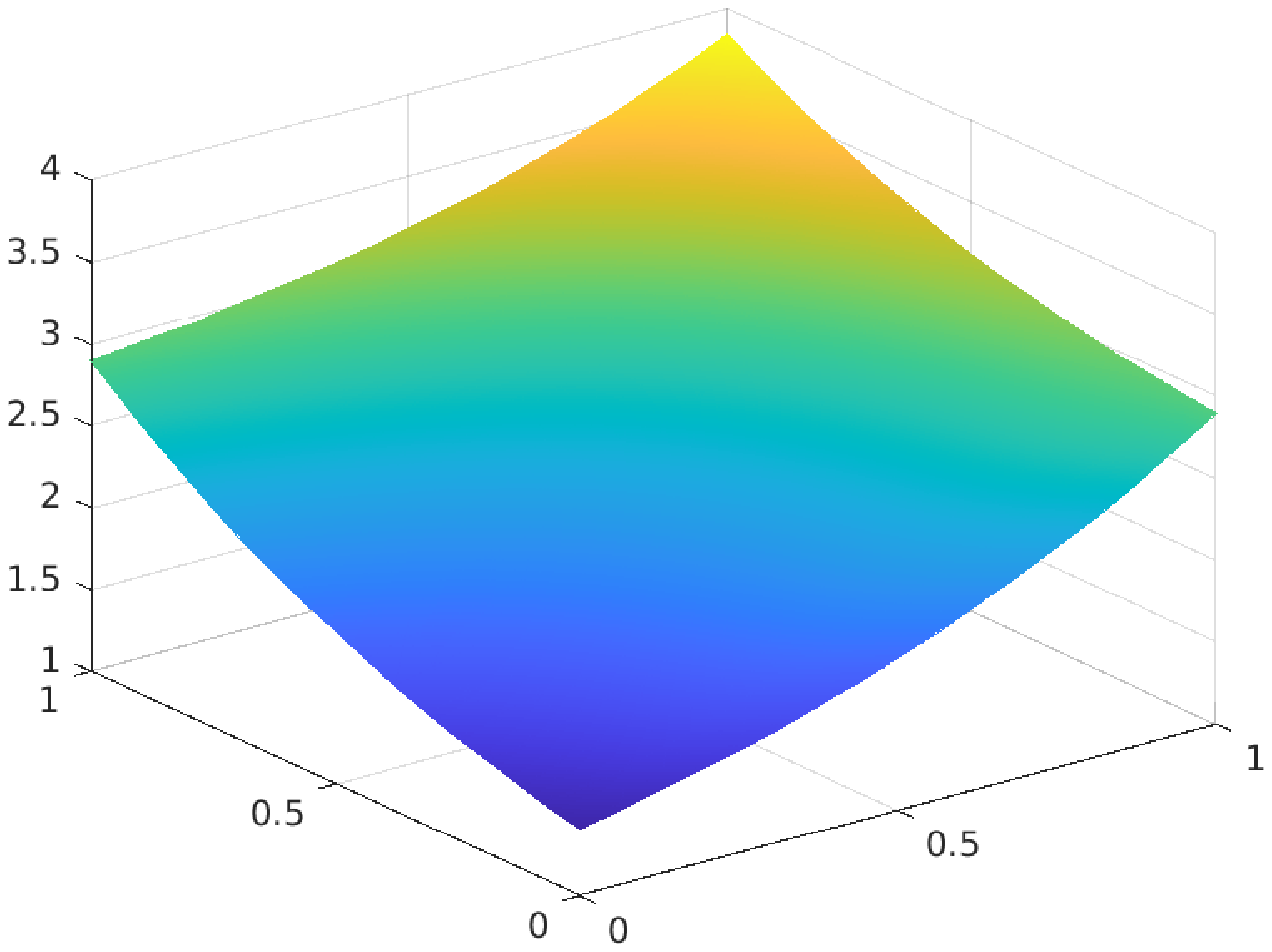}}
  \subfigure[Plot of $|\bm{u}_h|$ from \eqref{BREAFE}]{\includegraphics[width=6cm,height=6cm]{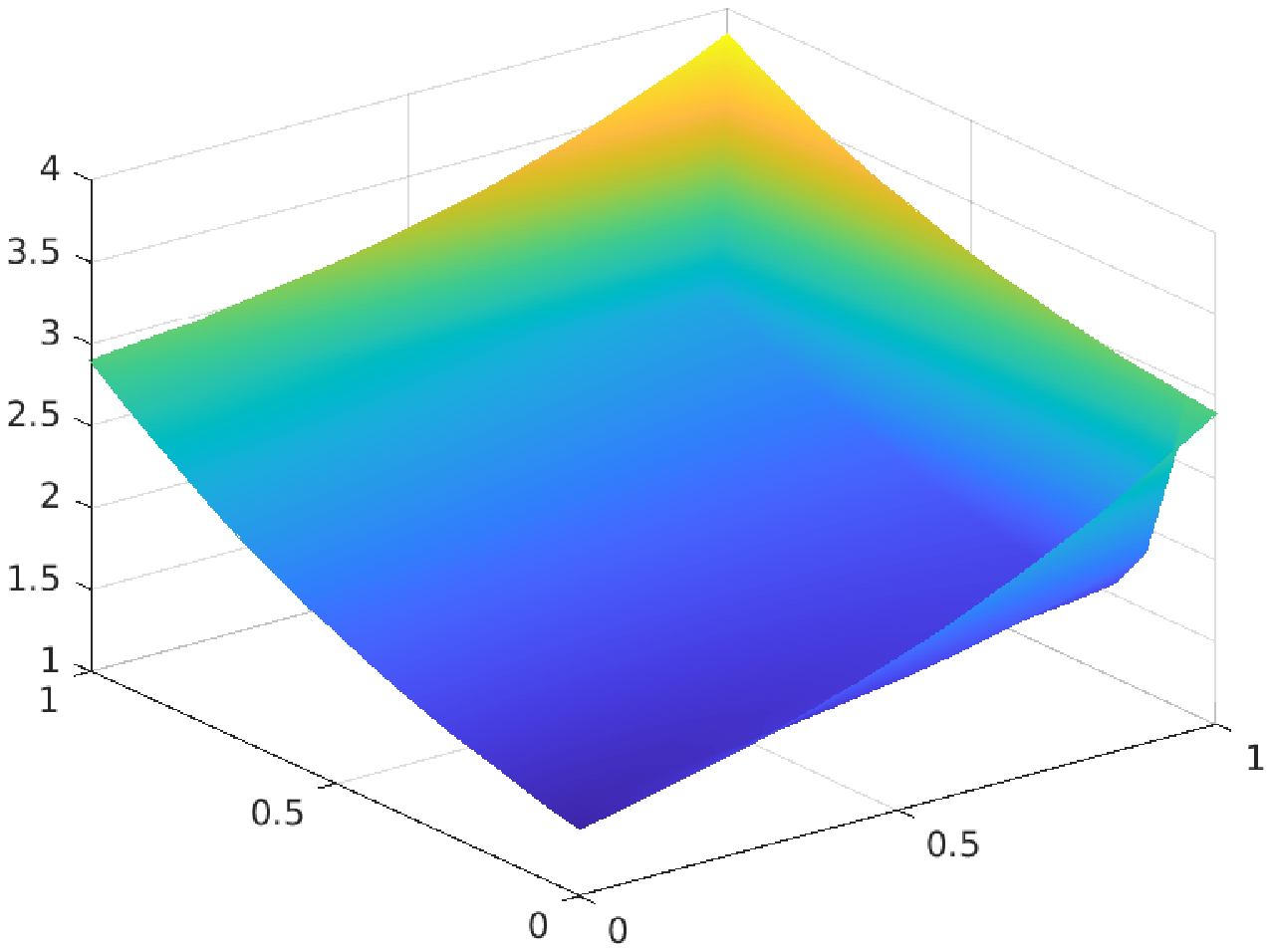}}
\end{tabular}
\caption{Numerical speeds for \eqref{Oseen} with $\nu=10^{-4}$, $\bm{b}=(10,1)^\top$ on a $16\times16$ uniform triangulation of $\Omega=[0,1]^2$.}
\label{Oseenplot}
\end{figure}

Now we compare the stabilized EAFE scheme \eqref{OseenBRbdmEAFE} with the unstabilized one \eqref{BREAFE}, where $\varepsilon=10^{-8}$.
Consider the Oseen problem \eqref{Oseen} with $\nu=10^{-4}$, $\bm{b}=(10,1)^\top$ and the exact solution \begin{align*}
    &\bm{u}(\bm{x})=\big(\exp(x_2), \exp(x_1)\big)^\top,\\
    &p(\bm{x})=\exp(x_1+x_2)-(\exp(1)-1)^2.
\end{align*}
We set $\Omega=[0,1]\times[0,1]$ to be the unit square. The domain $\Omega$ is partitioned into a  $16\times16$ uniform grid of right triangles. Numerical solutions are visualized in Fig.~\ref{Oseenplot}. In this example, the speed profile from \eqref{OseenBRbdmEAFE} is similar to the exponential exact speed $|\bm{u}|=\sqrt{\exp(2x_1)+\exp(2x_2)}$ while $|\bm{u}_h|$ from \eqref{BREAFE} exhibits an anomalous boundary layer.

\begin{remark}
For the convection-dominated elliptic problem $-\nabla \cdot(\alpha \nabla u+\bm{\beta} u)=f$, the classical EAFE scheme \cite{XuZikatanov1999} makes use of $-\alpha\Delta u$ without adding an artificial viscous term $-\varepsilon\Delta u$. Similarly, for nearly inviscid incompressible flows \eqref{Oseen} with $\nu\ll1,$
it is tempting to use $-\nu\Delta\bm{u}$ instead of adding the vanishing viscosity term $-\varepsilon\Delta\bm{u}$ when deriving EAFE schemes \eqref{OseenBRbdmEAFE}, \eqref{NSBRbdmEAFE}. However, an EAFE discretization for $-\nu\Delta\bm{u}+\bm{b}\cdot\nabla\bm{u}$ will not be Stokes stable because EAFE bilinear form does not take  face bubbles into account. 
\end{remark}

\begin{remark}
The EAFE scheme \eqref{OseenBRbdmEAFE} depends on the artificial diffusion constant $\varepsilon\ll1.$ In practice, there is no need to tune this parameter once $\varepsilon$ is sufficiently small. For example, schemes using $\varepsilon=10^{-8}$ or $\varepsilon=10^{-10}$ have almost the same performance for benchmark problems. 
\end{remark}

\section{Navier--Stokes equation}\label{secNS}
In view of the linear scheme \eqref{OseenBRbdmEAFE}, our method for the stationary Navier--Stokes equation \eqref{SNSE} seeks $(\bm{u}_h,p_h)\in\bm{V}_h\times Q_h$ satisfying
\begin{equation}\label{NSBRbdmEAFE}
  \begin{aligned}
  \nu a_h(\bm{u}_h,\bm{v}_h) + b_h(\bm{u}^l_h,\bm{v}_h;P_h\bm{u}^l_h)- (\nabla\cdot\bm{v}_h,p_h) &= (\bm{f},\Pi_h\bm{v}_h),\quad&&\forall~\bm{v}_h\in\bm{V}_h,\\
  (\nabla\cdot\bm{u}_h,q_h) &= 0,\quad&&\forall~q_h\in Q_h.
 \end{aligned}
\end{equation}
We recall that $P_h\bm{u}^l_h$ is the element-wise average of the linear component $\bm{u}^l_h$ given by $P_h\bm{u}^l_h|_T=\frac{1}{|T|}\int_T\bm{u}^l_hdx$. In practice, we use the fixed-point iteration to solve \eqref{NSBRbdmEAFE} by iterating the third argument in $b_h.$ At each iteration step, the linearized finite element scheme reduces to \eqref{OseenBRbdmEAFE}. The assembling of  stiffness matrices for linearized $b_h(\bullet,\bullet;P_h\bm{u}^l_h)$ could be easily done using \eqref{localbh}.

\subsection{Time-dependent problems}
In this subsection, we discuss the application of the stabilized  $P_1\times P_0$ EAFE method to the original evolutionary Navier--Stokes equation \eqref{NSE}. Applying the backward Euler method with time step-size  $\tau>0$ to discretize $\bm{u}_t$ and \eqref{NSBRbdmEAFE} to spatial variables, we obtain the fully discrete scheme
\begin{equation}\label{TD0NSBRbdmEAFE}
  \begin{aligned}
  &(\bm{u}_{h,n},\bm{v}_h)+\tau\nu a_h(\bm{u}_{h,n},\bm{v}_h) + \tau b_h(\bm{u}^l_{h,n},\bm{v}_h;P_h\bm{u}_{h,n}^l)\\
  &\quad- \tau(\nabla\cdot\bm{v}_h,p_{h,n}) = \tau(\bm{f}(t_n),\Pi_h\bm{v}_h)+(\bm{u}_{h,n-1},\bm{v}_h),\\
  &-\tau(\nabla\cdot\bm{u}^n_h,q_h) = 0,
 \end{aligned}
\end{equation}
for all $(\bm{v}_h,q_h)\in \bm{V}_h\times Q_h$,
where $t_n=n\tau$, $n=0,1,2,\ldots$, and $\bm{u}_{h,n}\approx\bm{u}(t_n)$, $p_{h,n}\approx p(t_n)$. However, the resulting algebraic system cannot be solved using the bubble reduction technique described in Section \ref{secOseen}. The reason is that the term $(\bm{u}_{h,n}^b,\bm{v}_h^b)$ introduces an extra mass matrix and the block  $\mathbf{A}_{bb}$ in \eqref{matrix} is no longer diagonal. To deal with such an issue, we consider the following quadrature on an element $T\in\mathcal{T}_h$:
\begin{equation}\label{quad}
    \int_Tfdx\approx\frac{|T|}{d+1}\sum_{F\in\mathcal{F}, F\subset\partial T}f(\bm{x}_F),
\end{equation}
where $\bm{x}_F$ is the barycenter of the face $F$. The formula \eqref{quad} is second-order in $\mathbb{R}^2$ and first-order in $\mathbb{R}^3$. We then introduce a discrete $L^2$ vector inner product
\begin{equation}
    (\bm{u},\bm{v})_h:=\frac{1}{d+1}\sum_{T\in\mathcal{T}_h}|T|\sum_{F\in\mathcal{F}, F\subset\partial T}\bm{u}(x_F)\cdot \bm{v}(x_F).
\end{equation}
Replacing $(\bm{u}_{h,n},\bm{v}_h), (\bm{u}_{h,n-1},\bm{v}_h)$ in \eqref{TD0NSBRbdmEAFE} with $(\bm{u}_{h,n},\bm{v}_h)_h, (\bm{u}_{h,n-1},\bm{v}_h)_h$, we arrive at the following modified  time-dependent scheme: Find $\{\bm{u}_{h,n}\}_{n\geq1}\subset\bm{V}_h$ and $\{p_{h,n}\}_{n\geq1}\subset Q_h$ such that
\begin{equation}\label{TD1NSBRbdmEAFE}
  \begin{aligned}
  &(\bm{u}_{h,n},\bm{v}_h)_h+\tau\nu a_h(\bm{u}_{h,n},\bm{v}_h) + \tau b_h(\bm{u}^l_{h,n},\bm{v}_h;P_h\bm{u}_{h,n}^l)\\
  &\quad- \tau(\nabla\cdot\bm{v}_h,p_{h,n}) = \tau(\bm{f}(t_n),\Pi_h\bm{v}_h)+(\bm{u}_{h,n-1},\bm{v}_h)_h,&&\quad\forall\bm{v}_h\in \bm{V}_h,\\
  &-\tau(\nabla\cdot\bm{u}_{h,n},q_h) = 0,&&\quad\forall q_h\in Q_h.
 \end{aligned}
\end{equation}
For two distinct faces $F\neq F^\prime\in\mathcal{F}_h,$ it holds that 
\begin{equation}
    (\phi_F\bm{n}_F,\phi_{F^\prime}\bm{n}_{F^\prime})_h=0.
\end{equation}
Therefore the matrix representation for  $(\bm{u}_{h,n}^b,\bm{v}_h^b)_h$ is a diagonal mass matrix and the linearized algebraic system for \eqref{TD1NSBRbdmEAFE} in fixed point iterations is of the form \eqref{matrix}, where $\mathbf{A}_{bb}$ is still a diagonal matrix  corresponding to $(\bm{u}_{h,n}^b,\bm{v}_h^b)_h+\tau\nu a_h(\bm{u}_{h,n}^b,\bm{v}_h^b)$. As explained in \eqref{matrix}, \eqref{condensed} in Section \ref{secOseen}, the computational cost of solving \eqref{TD1NSBRbdmEAFE} is comparable to a $[P_1]^d\times P_0$ element method.

Similarly to the stationary case, to enhance pressure-robustness, the work \cite{LinkeMerdon2016} replaces the terms  $(\bm{u}_{h,n},\bm{v}_h)$ and $(\bm{u}_{h,n-1},\bm{v}_h)$ in a backward Euler method with $(\Pi_h\bm{u}_{h,n},\Pi_h\bm{v}_h)$ and $(\Pi_h\bm{u}_{h,n-1},\Pi_h\bm{v}_h)$, respectively. For our purpose, we can only postprocess the test function $\bm{v}_h$ as the term $(\Pi_h\bm{u}_{h,n}^b,\Pi_h\bm{v}^b_h)$ would contribute a non-diagonal mass matrix. The resulting  time-dependent scheme is as follows: Find $\{\bm{u}_h^n\}_{n\geq1}\subset\bm{V}_h$ and $\{p_h^n\}_{n\geq1}\subset Q_h$ such that
\begin{equation}\label{TD2NSBRbdmEAFE}
  \begin{aligned}
  &(\bm{u}_{h,n},\Pi_h\bm{v}_h)_h+\tau\nu a_h(\bm{u}_{h,n},\bm{v}_h) + \tau b_h(\bm{u}^l_{h,n},\bm{v}_h;P_h\bm{u}^l_{h,n})\\
  &\quad- \tau(\nabla\cdot\bm{v}_h,p_{h,n}) = \tau(\bm{f}(t_n),\Pi_h\bm{v}_h)+(\bm{u}_{h,n-1},\Pi_h\bm{v}_h)_h,&&\quad\forall\bm{v}_h\in\bm{V}_h,\\
  &-\tau(\nabla\cdot\bm{u}_{h,n},q_h) = 0,&&\quad\forall q_h\in Q_h.
 \end{aligned}
\end{equation}
It follows from
direct calculation and \eqref{piphiene} that
\begin{equation}
    (\phi_F\bm{n}_F,\Pi_h(\phi_{F^\prime}\bm{n}_{F^\prime}))_h=c(\phi_F\bm{n}_F,\bm{\phi}^{\rm RT}_{F^\prime})_h=0,\quad \forall F\neq F^\prime\in\mathcal{F}_h,
\end{equation}
where $c=|F|/6$ in $\mathbb{R}^2$ and $c=|F|/15$ in $\mathbb{R}^3$.
Therefore the mass matrix from  $(\bm{u}_{h,n}^b,\Pi_h\bm{v}_h^b)_h$ is diagonal and \eqref{TD2NSBRbdmEAFE} could be solved in the same way as \eqref{TD1NSBRbdmEAFE} and \eqref{NSBRbdmEAFE}.

\begin{remark}
For Navier--Stokes equations with $\nu$ of moderate size, we could replace the EAFE form $b_h(\bm{u}_{h,n}^l,\bm{v}_h;P_h\bm{u}_{h,n}^l)$ with $(\bm{u}_{h,n}\cdot\nabla\bm{u}^l_{h,n},\Pi_h\bm{v}_h)$ and obtain a new stabilized $P_1\times P_0$ finite element method.
\end{remark}

\section{Numerical Experiments}\label{secNE}
In this section, we compare our schemes \eqref{OseenBRbdmEAFE}, \eqref{NSBRbdmEAFE}, \eqref{TD1NSBRbdmEAFE} with the standard BR finite element method using test functions postprocessed by the BDM interpolation $\Pi_h.$ Those postprocessing-based BR methods have been already argued in \cite{LinkeMatthiesTobiska2016,LinkeMerdon2016,JohnLinkeMerdonNeilanRebholz2017} to be superior to the classical ones. All experiments are performed in MATLAB R2020a and the linear solver is the MATLAB operation $\backslash$. For nonlinear problems, we use the fixed point iteration with 50 maximum number of iterations. The stopping criterion for nonlinear iterations is that the relative size of the increment $ |\bm{X}_{k+1}-\bm{X}_k|/|\bm{X}_k|$ is below $10^{-6}$, where $\bm{X}_k$ is the solution for the linear system at the $k$-th iteration step. The parameter $\varepsilon$ used in our EAFE-based schemes is set to be $10^{-10}$. 

\subsection{Convection-dominated Oseen problem}\label{expOseen}
First, we test the performance of \eqref{OseenBRbdmEAFE} and the classical scheme (cf.~\cite{LinkeMerdon2016,JohnLinkeMerdonNeilanRebholz2017})
\begin{equation}\label{BRbdm}
  \begin{aligned}
  \nu (\nabla\bm{u}_h,\nabla\bm{v}_h) + (\bm{b}\cdot\nabla\bm{u}_h,\Pi_h\bm{v}_h)- (\nabla\cdot\bm{v}_h,p_h) &= (\bm{f},\Pi_h\bm{v}_h),&&\quad\forall \bm{v}_h\in\bm{V}_h,\\
  (\nabla\cdot\bm{u}_h,q_h) &= 0,&&\quad\forall q_h\in Q_h
 \end{aligned}
\end{equation}
using the linear Oseen problem \eqref{Oseen} with $\nu=10^{-3},$ $\bm{f}=10(-x_2,x_1)^\top$, $\bm{g}=(0,0)^\top$,  $\bm{b}=(10,1)^\top$ on the unit square $\Omega=[0,1]\times[0,1]$. The mesh of $\Omega$ is relatively coarse and is shown in Fig.~\ref{Oseenexact}(a). We use the numerical velocity solution of  \eqref{OseenBRbdmEAFE} on a  relatively fine uniform mesh with $51200$ triangles as the exact velocity field, see Fig.~\ref{Oseenexact}(b).

\begin{figure}[ptb]
\centering
\begin{tabular}[c]{cccc}%
  \subfigure[The mesh used in Problem 5.1]{\includegraphics[width=6cm,height=5.5cm]{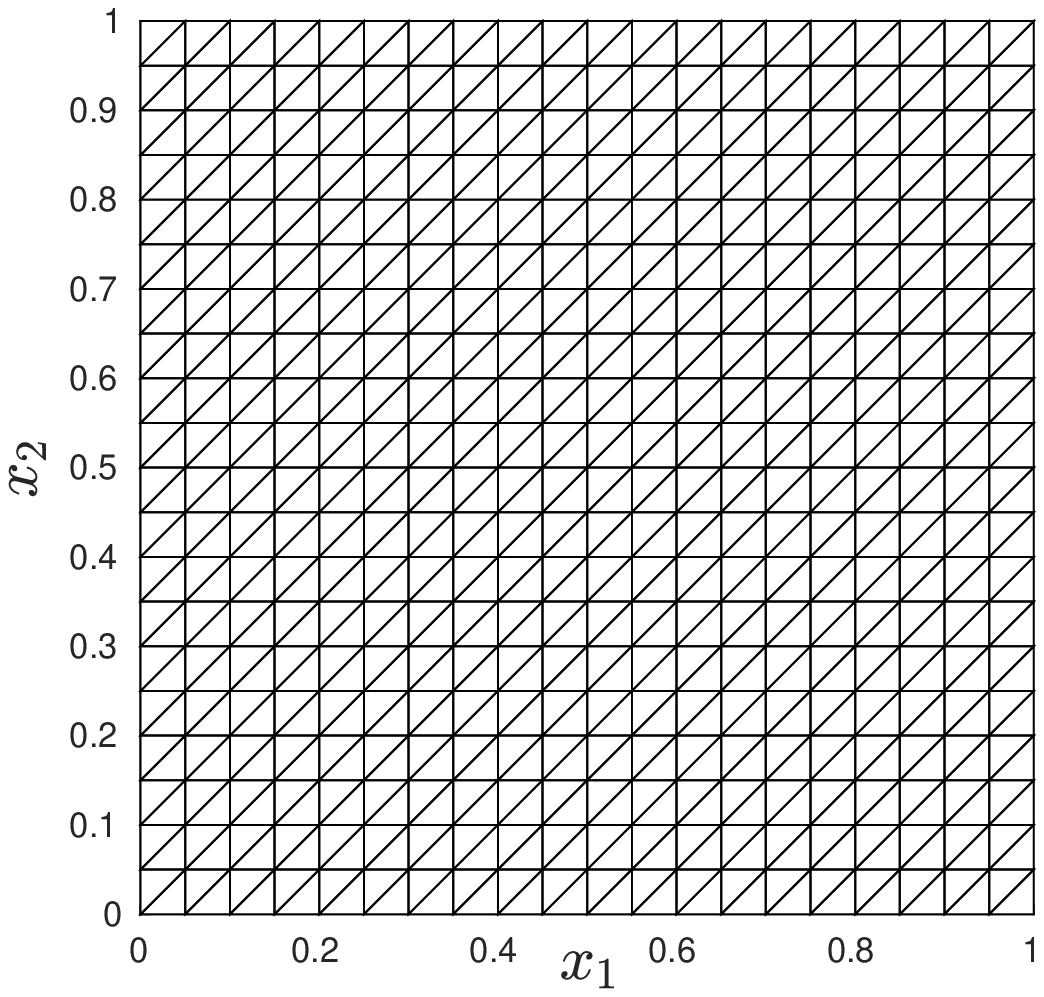}}\label{Oseenmesh}
  \subfigure[The exact velocity field in Problem 5.1]{\includegraphics[width=6cm,height=5.5cm]{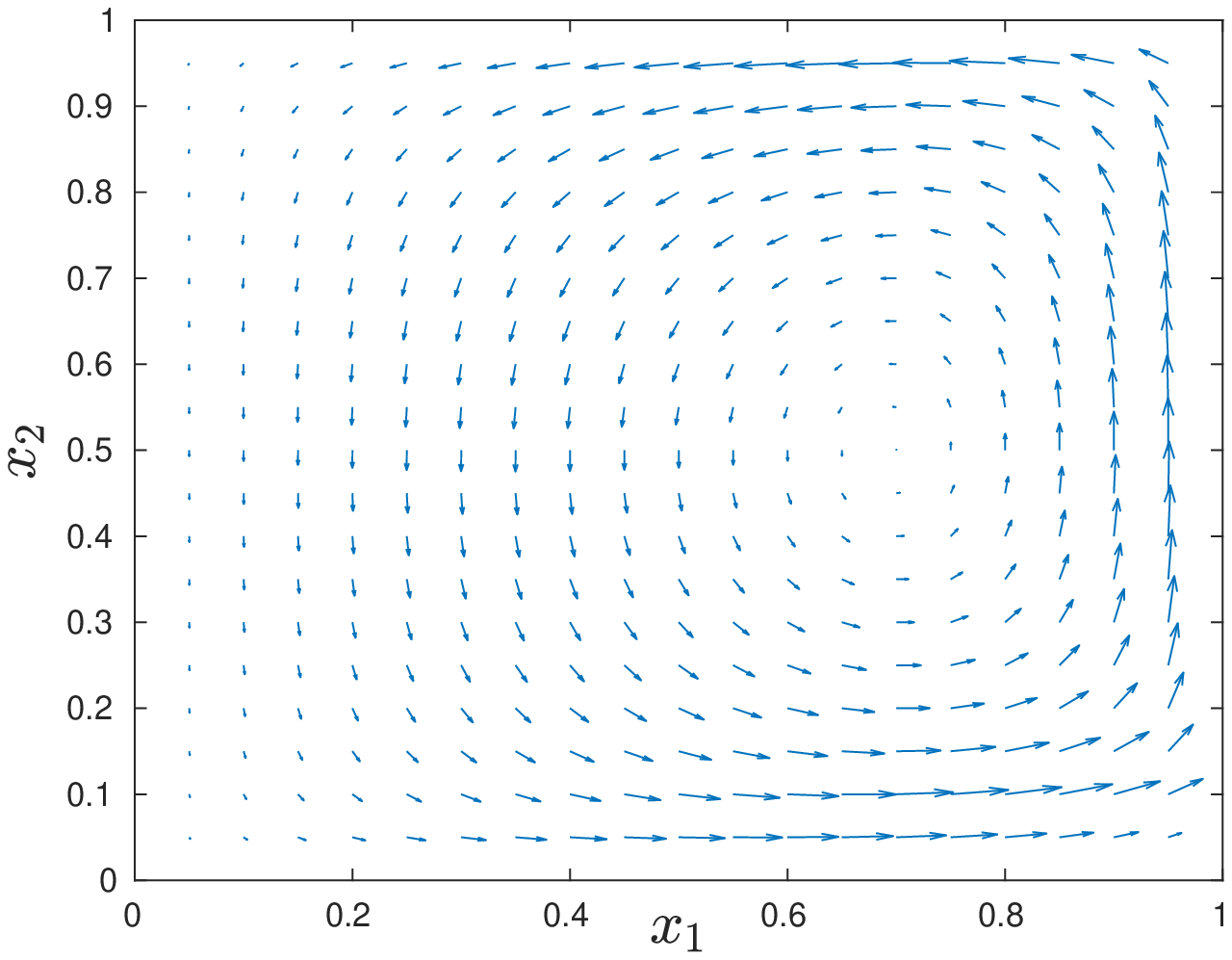}}
\end{tabular}
\caption{The grid and exact velocity field  for \eqref{Oseen} used in Problem \ref{expOseen}.}
\label{Oseenexact}
\end{figure}

\begin{figure}[ptb]
\centering
\begin{tabular}[c]{cccc}%
  \subfigure[Velocity field by the scheme \eqref{OseenBRbdmEAFE}]{\includegraphics[width=6cm,height=5.5cm]{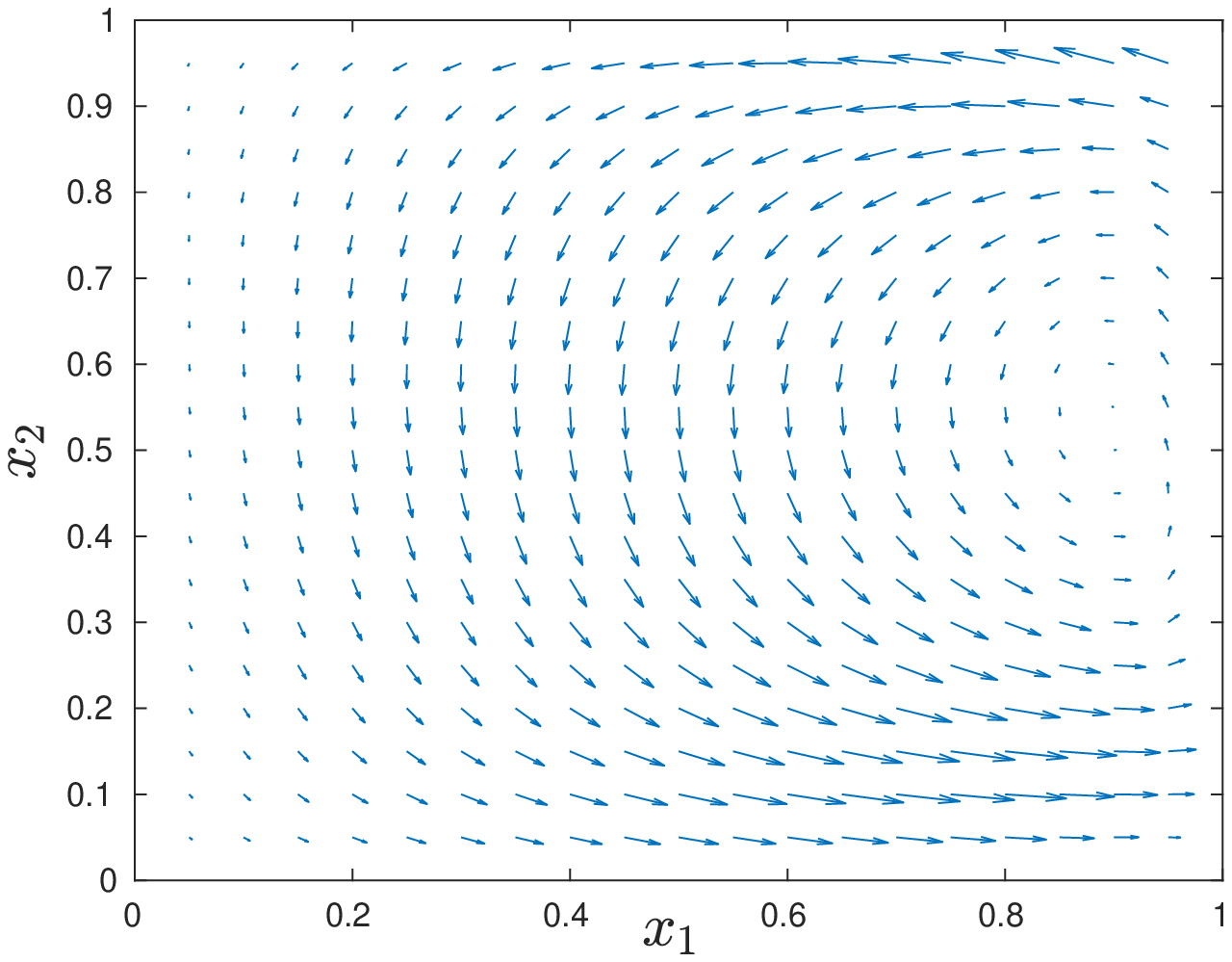}}
  \subfigure[Velocity field by the scheme \eqref{BRbdm}]{\includegraphics[width=6cm,height=5.5cm]{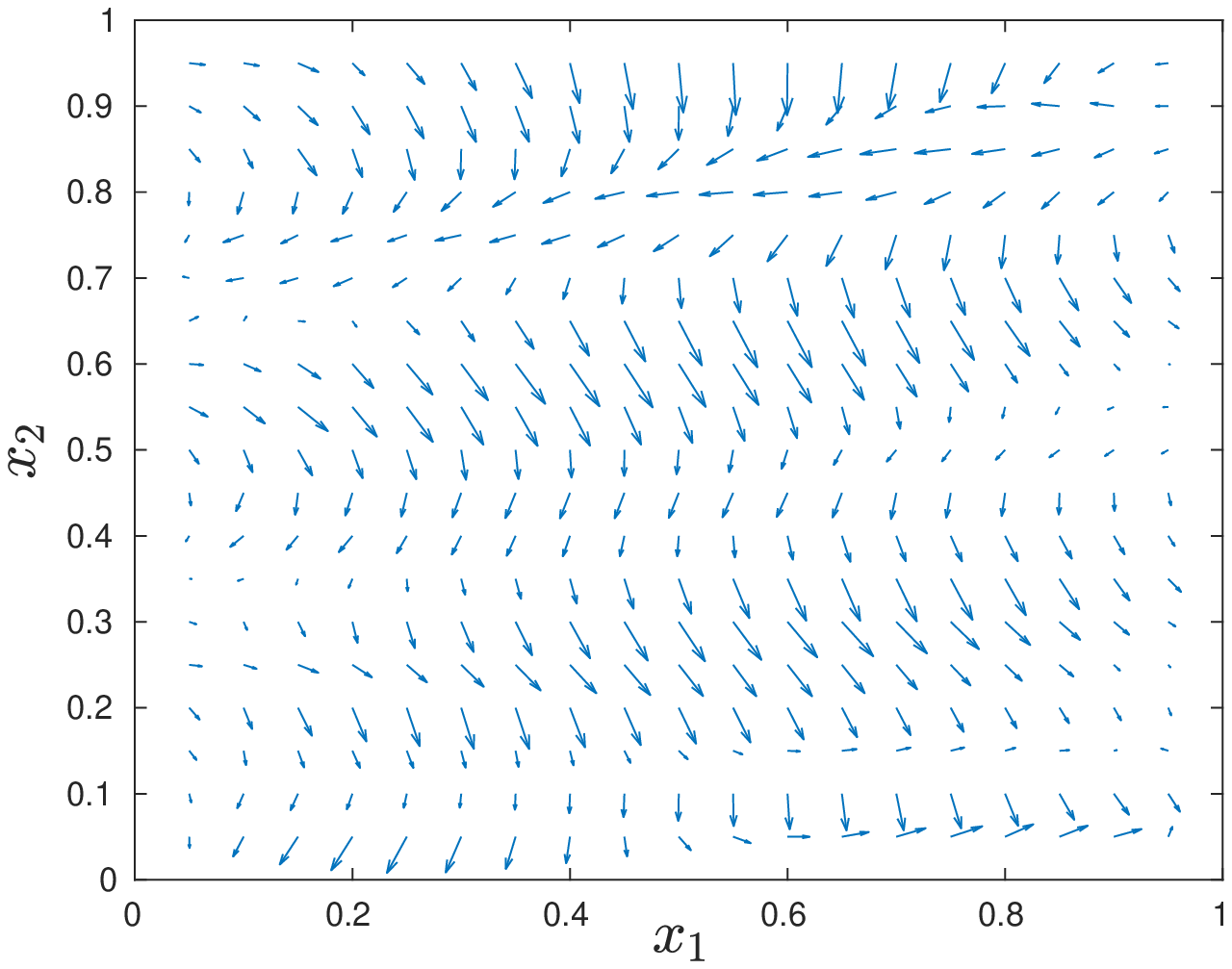}}
\end{tabular}
\caption{Numerical velocity fields of the schemes \eqref{OseenBRbdmEAFE} and \eqref{BRbdm} in Problem \ref{expOseen}}
\label{Oseenvelocity}
\end{figure}

Due to the sharp contrast between $\nu\ll|\bm{b}|$ and the homogeneous Dirichlet boundary condition, the exact velocity $\bm{u}$ is expected to have a sharp boundary layer. It could be observed from Fig.~\ref{Oseenvelocity}(b) that the classical scheme \eqref{BRbdm} yields a numerical velocity field with spurious oscillations in the low resolution setting. On the contrary, the stabilized EAFE scheme \eqref{OseenBRbdmEAFE} is able to produce a physically meaningful solution Fig.~\ref{Oseenvelocity}(a) on a coarse mesh, compared with the reference velocity profile in Fig.~\ref{Oseenexact}(b).

\subsection{Kovasznay flow}\label{expK}
\begin{table}[tbhp]
\caption{Convergence history of the stabilized $P_1\times P_0$ EAFE scheme \eqref{NSBRbdmEAFE} in Problem \ref{expK}}
\centering
\begin{tabular}{|c|c|c|c|c|c|c|}
\hline
${\rm error}\backslash \#{\rm dof}$ & 225
 &961
&  3969 & 16129 & 65025 &$\nu$\\ \hline
\hline
            
             $\|\bm{u}-\bm{u}^l_h\|$       &2.142	&5.715e-1	&1.448e-1 &3.712e-2&9.976e-3&$1$\\ \hline
             $\|p-p_h\|$       &3.105e+1	&1.995e+1	&1.070e+1
  &5.465&2.746&$1$\\ \hline
  $\|\bm{u}-\bm{u}^l_h\|$       &3.557e-1	&1.370e-1	&4.039e-2 &1.015e-2&3.309e-3&$10^{-3}$\\ \hline
             $\|p-p_h\|$       &2.151e-2	&1.840e-2	&1.048e-2
  &3.613e-3&1.567e-3&$10^{-3}$\\ \hline
    $\|\bm{u}-\bm{u}^l_h\|$       &3.560e-1	&1.481e-1	&4.496e-2 &9.396e-3&2.516e-3&$5\times10^{-4}$\\ \hline
             $\|p-p_h\|$       &1.317e-2	&1.793e-2 	&1.041e-2
  &2.676e-3&8.949e-4&$5\times10^{-4}$\\ \hline
      $\|\bm{u}-\bm{u}^l_h\|$       &3.589e-1	&1.701e-1	&9.665e-2 &1.033e-2&2.180e-3&$10^{-4}$\\ \hline
             $\|p-p_h\|$       &1.258e-2	&1.897e-2 	&1.905e-2
  &1.842e-3&4.462e-4&$10^{-4}$\\ \hline
\hline
\end{tabular}
\label{KvovasznayBRbdmEAFE}
\end{table}

\begin{table}[tbhp]
\caption{Convergence history of the classical scheme \eqref{NSBRbdm} in Problem \ref{expK}}
\centering
\begin{tabular}{|c|c|c|c|c|c|c|}
\hline
${\rm error}\backslash \#{\rm dof}$ & 401
 &1697
&  6977 & 28289 & 113921 &$\nu$\\ \hline
\hline
             $\|\bm{u}-\bm{u}^l_h\|$       &1.922	&4.436e-1	&1.058e-1 &2.584e-2&6.381e-3&$1$\\ \hline
             $\|p-p_h\|$       &2.886e+1	&1.606e+1	&8.293
  &4.183&2.095&$1$\\ \hline
  $\|\bm{u}-\bm{u}^l_h\|$       &1.057e+1	&2.006e+1	&4.569e-2 &7.543e-3&1.668e-3&$10^{-3}$\\ \hline
             $\|p-p_h\|$       &4.178	&3.125e+1	&1.242e-2
  &1.868e-3&4.786e-4&$10^{-3}$\\ \hline
    $\|\bm{u}-\bm{u}^l_h\|$       &1.002e+1	&9.190	&9.815 &9.131e-3&1.824e-3&$5\times 10^{-4}$\\ \hline
             $\|p-p_h\|$       &8.932	&8.246e+1	&5.006
  &2.277e-3&4.392e-4&$5\times 10^{-4}$\\ \hline
  $\|\bm{u}-\bm{u}^l_h\|$       &5.270	&1.126e+1	&1.651e+1 &1.871e+1&5.078&$10^{-4}$\\ \hline
             $\|p-p_h\|$       &9.940	&8.284	&1.970e+1
  &4.239e+1&3.174&$10^{-4}$\\ \hline
\hline
\end{tabular}
\label{KvovasznayBRbdm}
\end{table}

For the stationary Navier--Stokes  problem \eqref{SNSE},  we compare our scheme \eqref{OseenBRbdmEAFE} with the following classical scheme (cf.~\cite{LinkeMerdon2016,JohnLinkeMerdonNeilanRebholz2017})
\begin{equation}\label{NSBRbdm}
  \begin{aligned}
  \nu (\nabla\bm{u}_h,\nabla\bm{v}_h) + (\bm{u}_h\cdot\nabla\bm{u}_h,\Pi_h\bm{v}_h)- (\nabla\cdot\bm{v}_h,p_h) &= (\bm{f},\Pi_h\bm{v}_h),&&\quad\forall \bm{v}_h\in\bm{V}_h,\\
  (\nabla\cdot\bm{u}_h,q_h) &= 0,&&\quad\forall q_h\in Q_h
 \end{aligned}
\end{equation}
on the domain $\Omega=[-0.5,1.5]\times[0,2]$. The exact solution solutions of \eqref{SNSE} are taken to be
\begin{equation}\label{exactup}
    \bm{u}(\bm{x})=\begin{pmatrix}1-e^{\lambda x_1}\cos(2\pi x_2)\\\frac{\lambda}{2\pi}e^{\lambda x_1}\sin(2\pi x_2)\end{pmatrix},\quad p(\bm{x})=-\frac{1}{2}e^{2\lambda x_1}+\frac{1}{8\lambda}(e^{3\lambda}-e^{-\lambda}),
\end{equation}
where $\lambda=\frac{1}{2\nu}-\sqrt{\frac{1}{4\nu^2}+4\pi^2}$ and $\nu$ is varying.  In the literature, \eqref{exactup} is a benchmark problem known as the Kovasznay flow (cf.~\cite{DiPietroErn2012,ChenLiCorinaCimbala2020,ChenLi2021}). We start with a $8\times8$ uniform initial partition of $\Omega$ having 128 triangles and then refine each element in the current mesh by quad-refinement to obtain finer grids. The data shown from 2nd to 6th columns in Tables \ref{KvovasznayBRbdmEAFE} and \ref{KvovasznayBRbdm} are computed on the same mesh. Since the bubble component of $\bm{u}_h$ has little effect on the accuracy, we only consider the approximation property of the linear part $\bm{u}_h^l$.

Without dofs associated with faces, the size of algebraic systems from \eqref{NSBRbdmEAFE} is significantly smaller than \eqref{NSBRbdm}. 
In the case that $\nu=1$, the numerical accuracy of the classical scheme \eqref{NSBRbdm} is slightly better than the EAFE scheme \eqref{NSBRbdmEAFE}. As $\nu$ is increasingly small, our EAFE-stabilized $P_1\times P_0$ method is able to yield  numerical solutions with moderate accuracy even on the coarsest mesh. On the other hand, the performance of the classical method is not satisfactory on coarse meshes. In fact, the nonlinear iteration for \eqref{NSBRbdm} is not convergent unless the grid resolution is high enough.

\subsection{Evolutionary potential flow}\label{exp2dpotential}
In the rest of two experiments, we investigate the effectiveness of \eqref{NSBRbdmEAFE} and \eqref{TD1NSBRbdmEAFE} applied to benchmark potential flows  proposed by Linke\&Merdon \cite{LinkeMerdon2016}. Exact velocities of those problems are gradient of a harmonic polynomial. For potential flows, the pressure is relatively complicated and causes large velocity errors for numerical methods without pressure-robustness. Let $\chi(\bm{x},t)=t^2( 5x_1^4x_2-10x_1^2x_2^3+x_2^5 )$ be a polynomial that is  harmonic in space. We consider the evolutionary problem \eqref{NSE} with exact solutions
\begin{equation*}
    \bm{u}=\nabla\chi,\quad p=-\frac{|\bm{u}|^2}{2}-\chi_t+C,
\end{equation*}
where $C$ is a constant such that $\int_\Omega pdx=0.$ The corresponding load $\bm{f}=\bm{0}.$ The space domain $\Omega=[-0.5,0.5]\times[-0.5,0.5]$ and the time interval is $[0,2]$. We compare the EAFE $P_1\times P_0$ method \eqref{TD1NSBRbdmEAFE} with the following classical scheme (cf.~\cite{LinkeMerdon2016})
\begin{equation}\label{TDNSBRbdm}
  \begin{aligned}
  &(\Pi_h\bm{u}_{h,n},\Pi_h\bm{v}_h)+\tau\nu (\nabla\bm{u}_{h,n},\nabla\bm{v}_h) + \tau b_h(\bm{u}_{h,n}\cdot\nabla\bm{u}_{h,n},\Pi_h\bm{v}_{h})\\
  &\quad- \tau(\nabla\cdot\bm{v}_h,p_{h,n}) = \tau(\bm{f}(t_n),\Pi_h\bm{v}_h)+(\Pi_h\bm{u}_{h,n-1},\Pi_h\bm{v}_h),\quad&&\bm{v}_h\in \bm{V}_h,\\
  &-\tau(\nabla\cdot\bm{u}_{h,n},q_h) = 0,\quad &&q_h\in Q_h.
 \end{aligned}
\end{equation}
We set the time step-size to be $\tau=0.1$ and use a uniform criss-cross mesh  with 2048 right triangles. 

For the potential flow with $\nu=1$, the numerical performance of the classical scheme \eqref{TDNSBRbdm} and the EAFE scheme \eqref{TD1NSBRbdmEAFE} are comparable, see Table  \ref{2dpotentialtab1}. When $\nu=10^{-6}$ is exceedingly small, it is observed from Table \ref{2dpotentialtab2} that our stabilized $P_1\times P_0$ EAFE method outperforms the classical one. In particular, \eqref{TDNSBRbdm} stops converging after $t=1.5$ while our scheme \eqref{TD1NSBRbdmEAFE} maintains moderate accuracy and outputs a relatively good solution at $t=2$, see Fig.~\ref{2dpotentialvis}. 


\begin{figure}[!hptb]
\centering
\begin{tabular}[c]{cccc}%
  \subfigure[Speed profile at $t=2$ by the scheme \eqref{TD1NSBRbdmEAFE}]{\includegraphics[width=6cm,height=5.5cm]{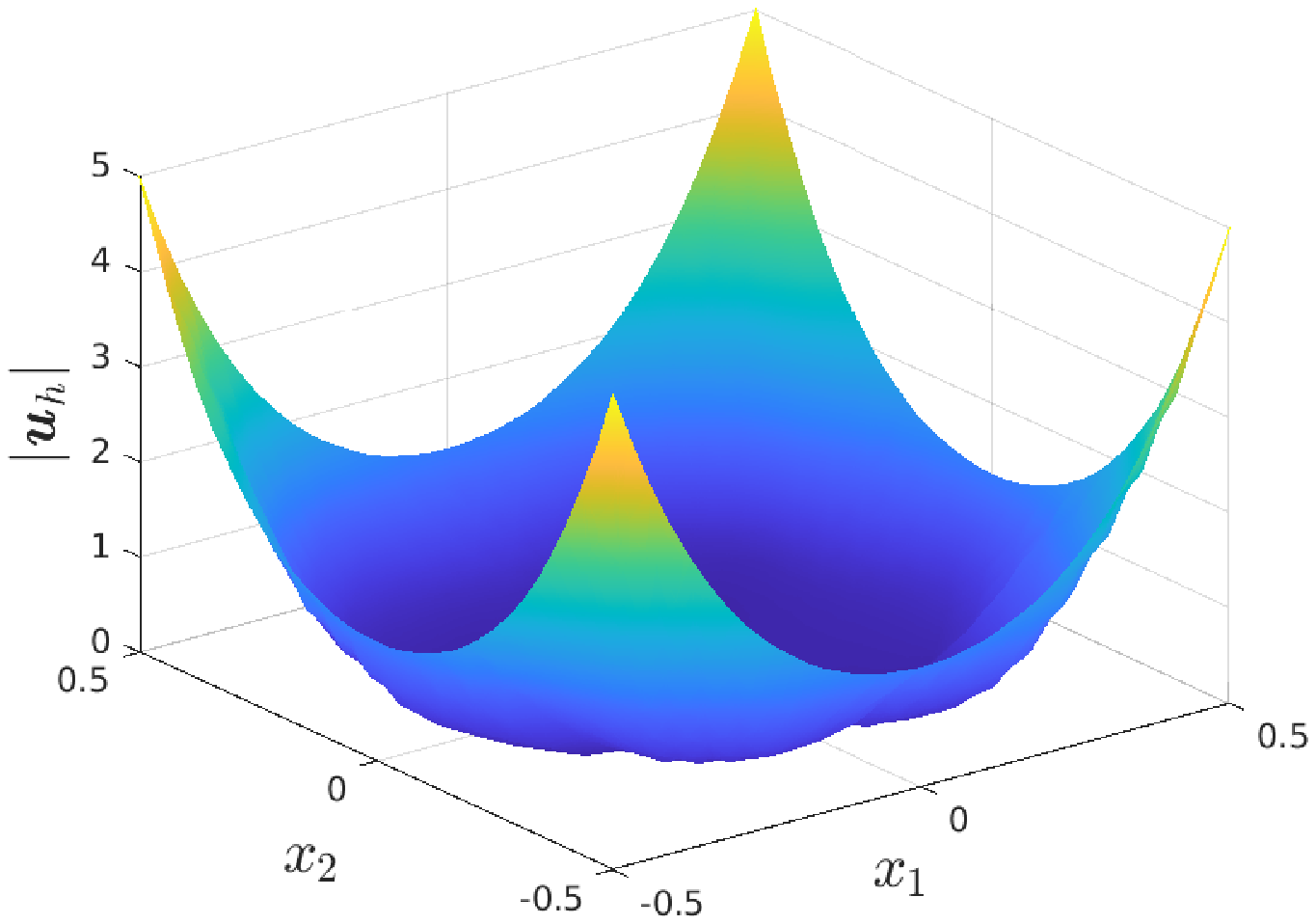}}
  \subfigure[Exact speed profile at $t=2$]{\includegraphics[width=6cm,height=5.5cm]{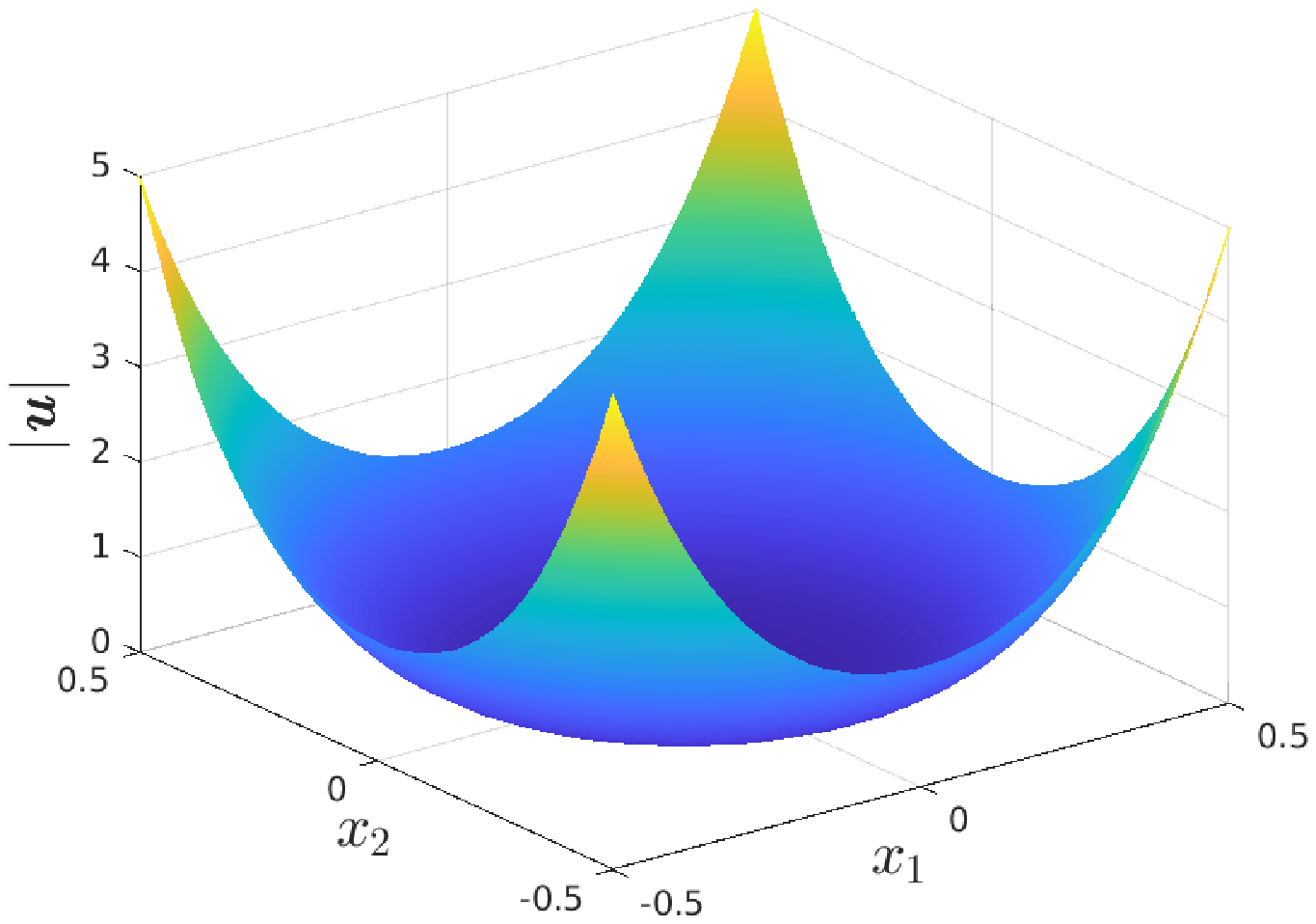}}
\end{tabular}
\caption{Numerical and exact speed profiles in Problem \ref{exp2dpotential}, $t=2,$ $\nu=10^{-6}$}
\label{2dpotentialvis}
\end{figure}

\begin{table}[!tbhp]
\caption{A comparison between the schemes \eqref{TD1NSBRbdmEAFE} and \eqref{TDNSBRbdm} in Problem \ref{exp2dpotential}, $\nu=1$.}
\centering
\begin{tabular}{|c|c|c|c|c|}
\hline
$t$ & $\|\bm{u}(t_n)-\bm{u}^l_{h,n}\|$ in \eqref{TD1NSBRbdmEAFE}
 &$\|p(t_n)-p_{h,n}\|$ in \eqref{TD1NSBRbdmEAFE}
&  $\|\bm{u}(t_n)-\bm{u}^l_{h,n}\|$ in \eqref{TDNSBRbdm} & $\|p(t_n)-p_{h,n}\|$ in \eqref{TDNSBRbdm}\\ \hline
\hline
             $0.5$       &4.640e-4	&1.371e-2	&3.890e-4 &2.220e-2\\ \hline
  $1.0$       &1.855e-3	&5.453e-2	&1.556e-3 &8.867e-2\\ \hline
             $1.5$       &4.164e-3	&1.323e-1	&3.501e-3
  &2.035e-1\\ \hline
    $2.0$       &7.387e-3	&2.838e-1	&6.224e-3 &3.805e-1 \\ \hline
\hline
\end{tabular}
\label{2dpotentialtab1}
\end{table}

\begin{table}[!tbhp]
\caption{A comparison between the schemes \eqref{TD1NSBRbdmEAFE} and \eqref{TDNSBRbdm} in Problem \ref{exp2dpotential}, $\nu=10^{-6}$.}
\centering
\begin{tabular}{|c|c|c|c|c|}
\hline
$t$ & $\|\bm{u}(t_n)-\bm{u}^l_{h,n}\|$ in \eqref{TD1NSBRbdmEAFE}
 &$\|p(t_n)-p_{h,n}\|$ in \eqref{TD1NSBRbdmEAFE}
&  $\|\bm{u}(t_n)-\bm{u}^l_{h,n}\|$ in \eqref{TDNSBRbdm} & $\|p(t_n)-p_{h,n}\|$ in \eqref{TDNSBRbdm}\\ \hline
\hline
             $0.5$       &2.893e-2	&1.401e-2	&1.962e-3 &3.193e-3\\ \hline
  $1.0$       &8.857e-2	&4.862e-2	&3.548e-2 &1.036e-2 \\ \hline
             $1.5$       &1.420e-1	&1.258e-1	&2.014e+1
  &3.372e+2\\ \hline
    $2.0$       &1.946e-1	&2.211e-1 	&1.085e+2&2.852e+3 \\ \hline
\hline
\end{tabular}
\label{2dpotentialtab2}
\end{table}

\subsection{3d potential flow}\label{exp3dpotential}

\begin{figure}[ptb]
\centering
\begin{tabular}[c]{cccc}%
  \subfigure[The tetrahedral mesh in Problem 5.4]{\includegraphics[width=6cm,height=5.5cm]{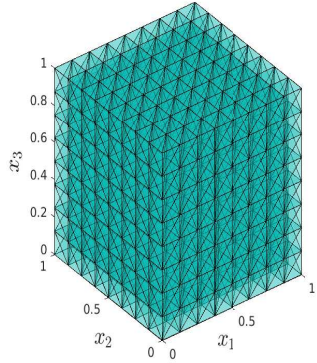}}
  \subfigure[Exact speed profile in Problem 5.4]{\includegraphics[width=6cm,height=5.5cm]{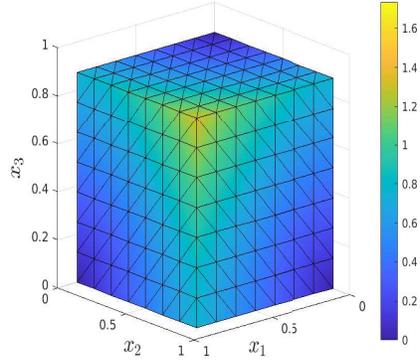}}
\end{tabular}
\caption{The mesh and exact speed in Problem \ref{exp3dpotential}}
\label{3dpotentialexact}
\end{figure}

\begin{figure}[ptb]
\centering
\begin{tabular}[c]{cccc}%
  \subfigure[Speed profile by the scheme \eqref{NSBRbdmEAFE}]{\includegraphics[width=6cm,height=5.5cm]{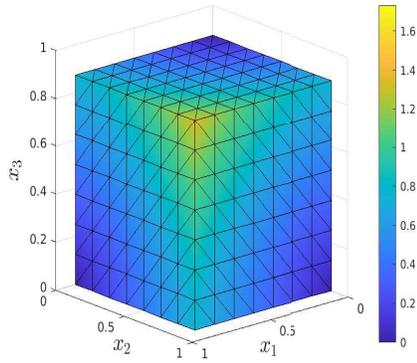}}
  \subfigure[Speed profile by the scheme \eqref{NSBRbdm}]{\includegraphics[width=6cm,height=5.5cm]{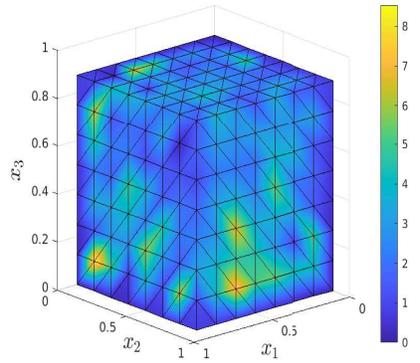}}
\end{tabular}
\caption{Speed profiles of the schemes \eqref{NSBRbdmEAFE} and \eqref{NSBRbdm} when $\nu=10^{-4}$ in Problem \ref{exp3dpotential}}
\label{3dpotentialvis}
\end{figure}

\begin{table}[tbhp]
\caption{A comparison between the schemes \eqref{NSBRbdmEAFE} and \eqref{NSBRbdm} in Problem \ref{exp3dpotential}}
\centering
\begin{tabular}{|c|c|c|c|c|}
\hline
$\nu$ & $\|\bm{u}-\bm{u}^l_h\|$ in \eqref{NSBRbdmEAFE}
 &$\|p-p_h\|$ in \eqref{NSBRbdmEAFE}
&  $\|\bm{u}-\bm{u}^l_h\|$ in \eqref{NSBRbdm} & $\|p-p_h\|$ in \eqref{NSBRbdm}\\ \hline
\hline
$1$       &2.205e-3	&4.393e-2	&2.201e-3 &4.037e-2\\ \hline
             $10^{-1}$       &2.234e-3	&2.925e-2	&2.220e-3 &2.813e-2\\ \hline
  $10^{-2}$       &2.605e-3	&2.872e-2	&2.192e-3 &2.799e-2\\ \hline
             $10^{-3}$       &4.639e-3	&2.837e-2	&2.192e-3
  &2.800e-2\\ \hline
    $10^{-4}$       &1.156e-2	&2.804e-2	&1.871 &1.074e+1\\ \hline
             $10^{-5}$       &1.607e-2	&2.806e-2	&1.064e+1
  &2.454e+3\\ \hline
\hline
\end{tabular}
\label{3dpotentialtab}
\end{table}

The last experiment is devoted to a 3 dimensional stationary Navier--Stokes problem \eqref{SNSE} on the unit cube $\Omega=[0,1]\times[0,1]\times[0,1]$, where the exact solution is a steady-state potential flow 
\begin{equation}
    \bm{u}(\bm{x})=\nabla(x_1x_2x_3),\quad p=-\frac{1}{2}\nabla|\bm{u}|^2+C,
\end{equation}
and $C$ is a constant such that $\int_\Omega pdx=0.$ The corresponding load  $\bm{f}=\bm{0}.$
The mesh of $\Omega$ is a uniform tetrahedral grid with 3072 elements, see Fig.~\ref{3dpotentialexact}(a). Numerical results are presented in Fig.~\ref{3dpotentialvis} and Table \ref{3dpotentialtab}.

For viscosity $\nu$ of moderate size, the performance of \eqref{NSBRbdmEAFE} and \eqref{NSBRbdm} are similar while \eqref{NSBRbdmEAFE} solves more economic algebraic linear systems with much less number of dofs. When $\nu\leq10^{-4}$, it is observed from Table \ref{3dpotentialtab} that the EAFE scheme \eqref{NSBRbdmEAFE} produces velocities and pressures of good quality on the fixed mesh while the fixed point iteration of \eqref{NSBRbdm} is indeed not convergent. Compared with \eqref{NSBRbdmEAFE}, the classical scheme \eqref{NSBRbdm} produces highly oscillating  solutions in the fixed point iteration when $\nu\leq10^{-4}$, see the visualization of $|\bm{u}_h|$ at the cross sections $x_1=0.8$, $x_2=0.8$, $x_3=0.8$ in Fig.~\ref{3dpotentialvis}.

\section{Concluding remarks}\label{seccon}
We have developed an EAFE-stabilized $P_1\times P_0$ finite element method for incompressible Navier--Stokes equations with small viscosity. For the Stokes problem, we have shown the robust a priori error analysis of our scheme with respect to $\nu$. It is straightforward to apply the technique in this paper to other Stokes element of the form $(P_1+{\it bubble})\times P_0$, see, e.g., the pointwise divergence-free Stokes elements in \cite{GuzmanNeilan2014,GuzmanNeilan2014b}. Moreover, we shall investigate stabilized $P_k\times P_{k-1}$ schemes based on reducing higher order $(P_k+{\it bubble})\times P^{\rm disc}_{k-1}$ Stokes elements \cite{GiraultRaviart1986,LinkeMatthiesTobiska2016,GuzmanNeilan2014,GuzmanNeilan2014b} or other technique \cite{DouglasWang1989} and higher order EAFE \cite{BankVassiZikatanov2017,2020WuZikatanov-a} in future research.

\providecommand{\bysame}{\leavevmode\hbox to3em{\hrulefill}\thinspace}
\providecommand{\MR}{\relax\ifhmode\unskip\space\fi MR }
\providecommand{\MRhref}[2]{%
  \href{http://www.ams.org/mathscinet-getitem?mr=#1}{#2}
}
\providecommand{\href}[2]{#2}

\end{document}